\documentclass[sn-mathphys]{sn-jnl}% Math and Physical Sciences Reference Style
\jyear{2021}%
\usepackage{float} %Forcing exact graphic placement - H
\usepackage{chngcntr} %Required package in order to use the macro.
\usepackage{caption}
\usepackage{subcaption}
\theoremstyle{thmstyleone}%
\newtheorem{theorem}{Theorem}%  meant for continuous numbers
\theoremstyle{thmstyletwo}%
\newtheorem{example}{Example}%
\newtheorem{remark}{Remark}%

\theoremstyle{thmstylethree}%
\newtheorem{definition}{Definition}%

\newtheorem{corollary}[theorem]{Corollary}
\newtheorem{lemma}[theorem]{Lemma}
\theoremstyle{definition}

\numberwithin{equation}{section}

\begin{document}

\title[Regularized split variational inclusions]{New iterative algorithms for solving split variational inclusions}


\author*[1]{\fnm{Soumitra} \sur{Dey}}\email{deysoumitra2012@gmail.com}
\author[2]{\fnm{Chinedu} \sur{Izuchukwu}}\email{chinedu.izuchukwu@wits.ac.za}
\author[1]{\fnm{Adeolu} \sur{Taiwo}}\email{taiwo.adeolu@yahoo.com}
\author[1]{\fnm{Simeon} \sur{Reich}}\email{sreich@technion.ac.il}
%\equalcont{}
\affil*[1]{\orgdiv{Department of Mathematics}, \orgname{The Technion- Israel Institute of Technology}, \orgaddress{\city{Haifa}, \postcode{3200003}, \country{Israel}}}
\affil[2]{\orgdiv{School of Mathematics}, \orgname{University of the Witwatersrand}, \orgaddress{\city{Private Bag 3, Johannesburg}, \postcode{2050}, \country{South Africa}}}

%\affil[]{\orgdiv{} \orgname{} \orgaddress{\street{} \city{}, \postcode{}, \state{}, \country{}}}


\abstract{In this paper we study a class of split variational inclusion (SVI) and regularized split variational inclusion (RSVI) problems in real Hilbert spaces.
We discuss various analytical properties of the net generated by the RSVI and establish the existence and uniqueness of the solution to the RSVI.
Using analytical properties of this net and under certain assumptions on the parameters and mappings associated with the SVI, we  establish the strong convergence of the sequence generated by our proposed iterative algorithm.
We also deduce another iterative algorithm by taking the regularization parameters to be zero in our proposed algorithm. We establish the weak convergence of the sequence generated by our new algorithm under certain assumptions. Moreover, we discuss two special cases of the SVI, namely the split convex minimization and the split variational inequality problems, and give several numerical examples.}




\maketitle

%================================================================================================================
\section{Introduction}
\label{Sec:1}

Let $\mathscr{H}$ be a real Hilbert space equipped with the inner product and the norm $\left\langle\cdot, \cdot\right\rangle$ and $\|\cdot\|$, respectively.
Let $\mathscr{A}:\mathscr{H}\rightarrow 2^{\mathscr{H}}$ be a set-valued mapping and $f:\mathscr{H}\rightarrow\mathscr{H}$ be a single-valued mapping.
Given these data, the {\em variational inclusion} (VI) problem is defined as follows:
find a point $x^*\in\mathscr{H}$ such that
\begin{align}\label{VInP}
0\in\mathscr{A}(x^*)+f(x^*).
\end{align}

Problem $(\ref{VInP})$ is quite ubiquitous. It provides a general and convenient framework for a unified study of the optimal solutions of many optimization problems arising in mathematical programming and optimal control. It also generalizes many important problems such as minimization problems (MP) and variational inequality problems (VIP).
In view of its versatility, problem (\ref{VInP}) has been well studied by several authors, who have presented many iterative algorithms for solving it numerically (see, for instance, \cite{SDEY2022, Chi, RT} references therein).

It is well known that when $\mathscr{A}$ is maximal monotone, then the VI (\ref{VInP}) reduces to the problem of finding a point $x^*\in H$ such that
\begin{align}\label{VInP:Fix}
x^*=J_{\lambda}^{\mathscr{A}}(I-\lambda f)x^*,
\end{align}
where $J_{\lambda}^\mathscr{A}=(I+\lambda\mathscr{A})^{-1}$ is the resolvent operator of the mapping $\mathscr{A}$ with parameter $\lambda>0$.

If $f=0$, then the above problem $(\ref{VInP})$ reduces to the inclusion problem studied by Rockafellar \cite{RTRO1976}. We have already mentioned that the VIP is a special case of problem (\ref{VInP}). Indeed, in addition, if $A(\cdot)=\partial i_C(\cdot)$, where $i_C(\cdot)$ is the indicator function of a nonempty, closed and convex subset $C$ of $\mathscr{H}$, and $\partial i_C$ is the subdifferential of $i_C(\cdot)$ (discussed later in detail), then problem $(\ref{VInP})$ reduces to the classical {\em variational inequality problem} (VIP) which has recently been studied, for example, by Thong et al. \cite{DVTH2022}, that is, $(\ref{VInP})$ reduces to the problem of finding a point $x^*\in C$ such that
\begin{align}\label{VIP}
\left\langle f(x^*), y-x^*\right\rangle\geq 0 \; \; \quad\forall  y\in C.
\end{align}

The variational inequality problem (VIP) is a useful tool in the optimization community. It was introduced and studied by Stampacchia \cite{GSTA1964} (in Euclidean spaces). Thereafter, many researchers have paid attention to it and have introduced several iterative algorithms for solving the VIP numerically (see, for instance,  \cite{YCEN2011, YCEN22011, XJCA2014, QLDO2019, QLDO2018, FFAC2003}).

At this point we recall that Censor et al. \cite[Section 2]{YCEN22012} introduced the general split inverse problem (SIP) in which there are given two vector spaces,
$X$ and $Y$, and a bounded linear operator $A:X\rightarrow Y$. In addition, two inverse problems are involved. The first one, denoted by IP$_{1}$, is
formulated in the space $X$ and the second one, denoted by IP$_{2}$, is formulated in the space $Y$. Given these data, the {\em Split Inverse Problem}
(SIP) is formulated as follows:%
\begin{gather}
\text{Find a point }x^{\ast }\in X\text{ that solves IP}_{1} \\
\text{such that}  \notag \\
\text{the point }y^{\ast }=Ax^{\ast }\in Y\text{ solves IP}_{2}.
\end{gather}%

The SIP is quite general because it enables one to obtain various split problems by making different choices of IP$_{1}$ and IP$_{2}$.
One important example is the split variational inclusion problem. Let $\mathscr{H}_1$ and $\mathscr{H}_2$ be two real Hilbert spaces. Let $f_1:\mathscr{H}_1\rightarrow\mathscr{H}_1$ and $f_2:\mathscr{H}_2\rightarrow\mathscr{H}_2$ be two single-valued mappings. Let $\mathscr{B}_1:\mathscr{H}_1\rightarrow 2^{\mathscr{H}_1}$ and $\mathscr{B}_2:\mathscr{H}_2\rightarrow 2^{\mathscr{H}_2}$ be two maximal monotone set-valued mappings and let $A:\mathscr{H}_1\rightarrow\mathscr{H}_2$ be a bounded linear operator. With these data, the {\em split variational inclusion} (SVI) problem is defined as follows: find a point $x^*\in\mathscr{H}_1$ such that
\begin{equation}\label{SVI}
\begin{cases}
 0\in\mathscr{B}_1(x^*)+f_1(x^*)\quad\text{and}\\
y^*=Ax^*\in \mathscr{H}_2\quad\text{solves}\quad 0\in\mathscr{B}_2(y^*)+f_2(y^*).
\end{cases}
\end{equation}

The SVI (\ref{SVI}) was first introduced by Moudafi \cite{AMOU2011}. Split variational inequality problems (SVIPs), split fixed point problems (SFPPs), as well as split feasibility problems (SFPs), can all be seen as special cases of the SVI problem. These can also be seen as well-known examples of the SIP
(see  \cite{YCEN22012,YCEN2006,YCEN2009,YCEN1994}).
In particular, if we take $f_1=f_2=0$, then the above problem (\ref{SVI}) reduces to a class of inclusion problems studied by Byrne et al. \cite{CBYR2012}. They introduced a $CQ$-type iterative algorithm (recall that the $CQ$ algorithm is a basic algorithm for solving the SFP) for solving such problems and established convergence theorems for it in infinite-dimensional Hilbert spaces. Such problems have recently been also studied by Hieu et al. \cite{DVHI2022}, who have proposed a proximal-like algorithm for solving it numerically.

Moudafi \cite{AMOU2011} studied the following iterative algorithm for solving problem (\ref{SVI}) in the case where $f_1$ and $f_2$ are inverse strongly monotone:
given  $z_1\in\mathscr{H}_1$ and $\gamma>0$, compute
\begin{equation}\label{algo:moudafi}
z_{n+1}=U(I-\gamma A^*(I-T)A)z_n \;  \quad \forall n\in\mathbb{N},
\end{equation}
where $\gamma\in (0, \frac{1}{L})$, $L$ is the spectral radius of $A^*A$, $A^*$ is the adjoint of $A$, $U=J_{\lambda}^{\mathscr{B}_1}(I-\lambda f_1)$, and $T=J_{\lambda}^{\mathscr{B}_2}(I-\lambda f_2)$. He established the weak convergence of the sequences generated by the above algorithm under certain assumptions on the parameters involved.

However, generally speaking, in an infinite-dimensional Hilbert space setting, the strong convergence of iterative algorithms is more desirable than weak convergence.
Therefore, it is natural to ask the following question: Is it possible to devise a simpler iterative algorithm than the one studied in \eqref{algo:moudafi} which will generate
strongly convergent approximating sequences? In this paper we give an affirmative answer to this question.

Let $\Omega$ be the solution set of the SVI (\ref{SVI}). In order to solve the SVI (\ref{SVI}) and to select a particular element in  $\Omega$, we focus our attention on solving the following variational inequality: find a point $u\in\Omega$ such that
\begin{align}\label{SVIVI}
\left\langle \mathscr{F}u, x^*-u\right\rangle\geq 0 \;  \quad \forall x^*\in\Omega,
\end{align}
where $\mathscr{F}:\mathscr{H}_1\rightarrow\mathscr{H}_1$ is a single-valued mapping.

Problem $(\ref{SVIVI})$ can be viewed as a bilevel variational inequality problem.

In this paper we study the SVI (\ref{SVI}) under the following assumptions:

\begin{enumerate}
\item[(A1)] $\mathscr{B}_1$ and $\mathscr{B}_2$ are maximal monotone

\item[(A2)] $f_1$ and $f_2$ are $\tau_1$ and $\tau_2$-inverse strongly monotone (ism), respectively.

\item[(A3)] $\mathscr{F}$ is $\gamma$-strongly monotone and $L$-Lipschitz continuous.

\item[(A4)] The solution set $\Omega$ of the problem SVI (\ref{SVI}) is nonempty.
\end{enumerate}

Note that $\Omega$ is a subset of the Hilbert space $\mathscr{H}_1$. If $x^*\in\Omega$, then $y^* = Ax^*$ satisfies the second inclusion of (\ref{SVI}).

Our algorithm is based on the regularization technique, which provides the strong convergence of the approximating sequences generated by the proposed algorithm due to the special properties of the single-valued mapping $\mathscr{F}$.

This paper is organized as follows. In Section \ref{Sec:2} we collect some basic definitions and results. In Section \ref{Sec:3} we establish several analytical properties of the net generated by the RSVI (\ref{RSVI}). In Section \ref{Sec:4} we introduce an algorithm for solving the SVI \eqref{SVI} and prove its strong convergence under certain assumptions on the parameters involved. In Section \ref{Sec:5} we have discussed two special cases, namely the split convex minimization problem (SCMP) and the split variational inequality problem (SVIP) as applications of the SVI \eqref{SVI}. In Section \ref{Sec:6} we have illustrated the efficiency of our algorithms numerically.

%=======================================================================================================
\section{Preliminaries}\label{Sec:2}
\noindent
In this section we collect some basic definitions and results, which we use in order to prove our main results.

\begin{definition}\cite{HHBA2011}
A single-valued mapping $f:\mathscr{H}\rightarrow\mathscr{H}$ is said to be:
\begin{enumerate}
\item $\gamma$-strongly monotone if there exists $\gamma>0$ such that
\begin{align*}
\left\langle  f(x)-f(y), x-y\right\rangle\geq\gamma\|x-y\|^2 \; \quad\forall x, y\in\mathscr{H}.
\end{align*}

\item $\eta$-inverse strongly monotone ($\eta$-ism, for short) if there exists $\eta>0$ such that
\begin{align*}
\left\langle  f(x)-f(y), x-y\right\rangle\geq\eta\|f(x)-f(y)\|^2 \; \quad\forall x, y\in\mathscr{H}.
\end{align*}

\item firmly nonexpansive if
\begin{align*}
\|(I-f)(x)-(I-f)(y)\|^2+\|f(x)-f(y)\|^2\leq \|x-y\|^2 \; \quad\forall x, y\in\mathscr{H},
\end{align*}
or equivalently,
\begin{align*}
\left\langle  f(x)-f(y), x-y\right\rangle\geq\|f(x)-f(y)\|^2 \; \quad\forall x, y\in\mathscr{H},
\end{align*}
or equivalently, $f$ is $1$-ism.

\item $\alpha$-averaged if there exists $\alpha\in(0, 1)$ and a nonexpansive mapping $S:\mathscr{H}\rightarrow\mathscr{H}$ such that
\begin{align*}
f=(1-\alpha)I+\alpha S.
\end{align*}
\end{enumerate}
\end{definition}

\begin{definition}\cite{AJKU2005}
A function $f:\mathscr{H}\rightarrow\mathbb{R}\cup\{+\infty\}$ is lower semi-continuous at a point $x^*\in\mathscr{H}$ if for any sequence $\{x_n\}\subset \mathscr{H}$ with $x_n\rightarrow x^*$, we have
\begin{align*}
f(x^*)\leq\liminf_{n\rightarrow\infty} f(x_n).
\end{align*}
\end{definition}

If the function $f:\mathscr{H}\rightarrow\mathbb{R}\cup\{+\infty\}$ is lower semi-continuous at each $x^*\in\mathscr{H}$, then it is called lower semi-continuous.

\begin{definition}\cite{UMOS1967}
Let $C$ be a nonempty convex subset of $\mathscr{H}$. A single-valued mapping $f:C\rightarrow\mathscr{H}$ is said to be hemicontinuous if for any $x,y\in C$ and $z\in\mathscr{H}$, the function $t\rightarrow\left\langle f(x+t(y-x)),z\right\rangle$ from $[0,1]$ into $\mathbb{R}$ is continuous. 
%Thant is,
%\begin{align*}
%\lim_{t\rightarrow t_0\in [0,1]} \left\langle f(x+t(y-x)), z\right\rangle= \left\langle f(x+t_0(y-x)), z\right\rangle.
%\end{align*}
\end{definition}

\begin{remark}
Every Lipschitz continuous mapping is hemicontinuous.
\end{remark}

The graph of a set-valued mapping $\mathscr{A}:\mathscr{H}\rightarrow 2^\mathscr{H}$ is denoted by $gra(\mathscr{A})$ and is defined by $gra(\mathscr{A}):=\left\lbrace (x, u)\in\mathscr{H}\times\mathscr{H}: u\in\mathscr{A}(x)\right\rbrace.$

\begin{definition}\cite{HHBA2011}
Let $\mathscr{A}: \mathscr{H}\rightarrow 2^\mathscr{H}$ be a set-valued mapping. Then $\mathscr{A}$ is said to be:
\begin{enumerate}
\item monotone if for every $x,y\in\mathscr{H}$, we have
\begin{align*}
\left\langle u-v, x-y\right\rangle\geq 0 \; \quad\forall u\in\mathscr{A}(x), v\in\mathscr{A}(y).
\end{align*}
\item maximal monotone if it is monotone and there does not exist any other monotone mapping $\mathscr{B}:\mathscr{H}\rightarrow 2^\mathscr{H}$ such that $gra(\mathscr{A})\subsetneq gra(\mathscr{B})$.
\end{enumerate}
\end{definition}

\begin{lemma}\cite{HHBA2011}
If $\mathscr{A}:\mathscr{H}\rightarrow 2^\mathscr{H}$ is a maximal monotone set-valued mapping, then the resolvent $J_{\lambda}^\mathscr{A}=(I+\lambda\mathscr{A})^{-1}$, where $\lambda>0$ is a constant, of $\mathscr{A}$ is a single-valued and firmly nonexpansive mapping.
\end{lemma}

\begin{lemma}\label{Lem:summax}\cite{BLEM1997}
Let $\mathscr{A}:\mathscr{H}\rightarrow 2^\mathscr{H}$ be a maximal monotone set-valued mapping and let $f:\mathscr{H}\rightarrow\mathscr{H}$ be a monotone and Lipschitz continuous single-valued mapping. Then $\mathscr{B}=\mathscr{A}+f$ is also maximal monotone.
\end{lemma}

\begin{lemma}\cite{HHBA2011}
For any $x, y\in\mathscr{H}$ and $\lambda\in\mathbb{R}$, the following identity holds:
\begin{align*}
\|(1-\lambda)x+\lambda y\|^2=(1-\lambda)\|x\|^2+\lambda\|y\|^2-\lambda(1-\lambda)\|x-y\|^2.
\end{align*}
\end{lemma}

\begin{lemma}(Young's Inequality)\label{Lem:Young}
For any $x,y\in\mathscr{H}$ and any $\epsilon>0$, we have
\begin{align*}
\left\langle x,y\right\rangle\leq\frac{\epsilon\|x\|^2}{2}+\frac{\|y\|^2}{2\epsilon}.
\end{align*}
\end{lemma}

\begin{lemma}\label{Lem:ism:firmly}\cite[Lemma 3.2]{WTAK2015}
\begin{enumerate}
\item The composition of finitely many averaged mappings is also averaged..

\item If $T:\mathscr{H}\rightarrow\mathscr{H}$ is $\beta$-ism, then for $\gamma>0$, $\gamma T$ is $(\beta/\gamma)$-ism.

\item $T:\mathscr{H}\rightarrow\mathscr{H}$ is averaged if and only if its complement $I-T$ is $\beta$-ism for some $\beta > 1/2$. Indeed, for $\alpha\in (0, 1)$, $T$ is $\alpha$-averaged if and only if $I-T$ is $(1/2\alpha)$-ism.

%\item If $T=(1-\beta)R+\beta S$ for some $\beta\in(0, 1)$ and $R:\mathscr{H}\rightarrow\mathscr{H}$ be $\alpha-$averaged for some $\alpha\in (0, 1)$ and $S:H\rightarrow H$ be a nonexpansive mapping then $T$ is $\beta+(1-\beta)\alpha$- averaged.
\end{enumerate}
\end{lemma}

\begin{lemma}\label{Lem:ism}\cite{WTAK2015}
Let $\mathscr{H}_1$ and $\mathscr{H}_2$ be two Hilbert spaces, $A:\mathscr{H}_1\rightarrow \mathscr{H}_2$ be a bounded linear operator such that $A\neq 0$,
and let $T:\mathscr{H}_2\rightarrow \mathscr{H}_2$ be a nonexpansive mapping. Then $A^*(I-T)A$ is $\frac{1}{2\|A\|^2}$-ism, where $A^*$ is the adjoint of $A$.
\end{lemma}

\begin{lemma}\label{Lem:VIP}\cite{IYAM2001}
Let $C$ be a nonempty, closed and convex subset of $\mathscr{H}$. If $\mathscr{F}: \mathscr{H} \rightarrow\mathscr{H}$ is an $L$-Lipschitz continuous and $\beta$-strongly monotone single-valued mapping over $C$, then the VIP
\begin{align}
\left\langle\mathscr{F}(x^*), y-x^*\right\rangle\geq 0 \; \quad\forall y\in C\nonumber
\end{align}
has a unique solution $x^*\in C$.
\end{lemma}

\begin{lemma}\label{Lem:Minty}\cite[Minty Lemma]{FEBR1966,ABEH1997}
Let $C$ be a nonempty, closed and convex subset of $\mathscr{H}$. Let $\mathscr{F}:C\rightarrow\mathscr{H}$ be a hemicontinuous and monotone single-valued mapping.
Then $u \in C$ is a solution of the variational inequality
\begin{align*}
\left\langle\mathscr{F}u, y-u\right\rangle\geq 0 \; \quad\forall y\in C
\end{align*}
if and only if it is a solution of the problem
\begin{align*}
\left\langle\mathscr{F}y, y-u\right\rangle\geq 0 \;  \quad\forall y\in C.
\end{align*}
\end{lemma}

\begin{lemma}\cite{HHBA2011}\label{Xu:lem}
Let $K$ be a nonempty, closed and convex subset of a Hilbert space $\mathscr{H}$. Let $\{x_n\}$ be a sequence which enjoys the following two properties:
\begin{enumerate}
\item  $\lim\limits_{n\rightarrow \infty}\|x_n - x\|$ exists for each $x\in K$.
\item every subsequential weak limit point of $\{x_n\}$ lies in $K$.
\end{enumerate}
Then $\{x_n\}$ converges weakly to a point in $K$.	
\end{lemma}

\begin{lemma}\label{Lem:Lastconvergence}\cite{SSAE2012}
Let $\left\lbrace \Gamma_n\right\rbrace$ be a sequence of nonnegative real numbers, $\left\lbrace \alpha_n\right\rbrace$ be a sequence of real numbers in $(0, 1)$ satisfying the condition $\sum_{n=1}^{\infty}\alpha_n=+\infty$,
and let $\left\lbrace b_n\right\rbrace$ be a sequence of real numbers. Assume that
\begin{align*}
\Gamma_{n+1}\leq(1-\alpha_n)\Gamma_n+\alpha_n b_n \; \quad \forall n\geq 1.
\end{align*}
If $\limsup_{k\rightarrow\infty}b_{n_k}\leq 0$ for every subsequence $\left\lbrace \Gamma_{n_k}\right\rbrace$ of $\left\lbrace \Gamma_n\right\rbrace$ satisfying
\begin{align*}
\liminf_{k\rightarrow\infty} (\Gamma_{n_{k}+1}-\Gamma_{n_k})\geq 0,
\end{align*}
then $\lim_{n\rightarrow\infty} \Gamma_n=0.$
\end{lemma}

\begin{lemma}\label{Lem:LipIsm}\cite[Baillon-Haddad]{HHBA2010}
Let $f:\mathscr{H}\rightarrow\mathbb{R}$ be convex and Fr\'{e}chet differentiable on $\mathscr{H}$. If $\nabla f$ is $\beta$-Lipschitz continuous for some $\beta>0$,
then $\nabla f$ is $1/\beta$-ism.
\end{lemma}

%=======================================================================================================
\section{Properties of an approximating net}
\label{Sec:3}
In this section we establish several properties of the net $\{x_\alpha\}$ generated by the regularized split variational inclusion (\ref{RSVI}) (see below).

Note that if $f_1:\mathscr{H}_1\rightarrow \mathscr{H}_1$ and  $f_2:\mathscr{H}_2\rightarrow\mathscr{H}_2$ are $\tau_1$-ism and $\tau_2$-ism, respectively,
then $f_1$ and $f_2$ are $\tilde{\tau}$-ism, where $\tilde{\tau} = \min\{\tau_1, \tau_2\}$.

It is well known that in general problem (\ref{SVIVI}) is ill posed. However, it is evident that under assumptions $(A3)$ and $(A4)$ problem (\ref{SVIVI}) becomes well posed.
That is, under assumptions $(A3)$ and $(A4)$, problem (\ref{SVIVI}) has a unique solution (by Lemma \ref{Lem:VIP}).

Since $f_2$ is $\tilde{\tau}$-ism, we conclude, using Lemma \ref{Lem:ism:firmly}, that $\lambda f_2$ is $\tilde{\tau}/\lambda$-ism for $\lambda>0$.
This implies that $I-\lambda f_2$ is averaged for $\lambda\in (0, 2\tilde{\tau})$. Therefore, for $\lambda\in (0, 2\tilde{\tau})$, the mapping $T=J_\lambda^{\mathscr{B}_2}(I-\lambda f_2)$ is single-valued and averaged. Hence $T$ is  a single-valued nonexpansive mapping for $\lambda\in (0, 2\tilde{\tau})$.
This implies, by Lemma \ref{Lem:ism}, that $S=A^*(I-T)A$ is $\frac{1}{2\|A\|^2}$-ism. Therefore $S$ is monotone and Lipschitz continuous.

Next, we prove the following result.
\begin{lemma}
Suppose that Assumptions $(A1)$--$(A4)$ hold and that $\lambda\in (0,2\tilde{\tau})$. Then solving the SVI \eqref{SVI} is equivalent to solving the following inclusion:

Find $\tilde{x}\in\mathscr{H}_1$ such that
\begin{equation}\label{GSVI}
0\in \mathscr{B}_1(\tilde x)+f_1(\tilde x)+A^*(I-J_\lambda^{\mathscr{B}_2}(I-\lambda f_2))A(\tilde x).
\end{equation}

\begin{proof}
Let $x^*\in\Omega$. Then we know that
\begin{align*}
& x^*=J_\lambda^{\mathscr{B}_1}(I-\lambda f_1)x^*\text{ and }\\&
Ax^*=J_\lambda^{\mathscr{B}_2}(I-\lambda f_2)A(x^*).
\end{align*}
Using these facts, we see that $x^*\in\mathscr{H}_1$ solves problem \eqref{GSVI}.

Now let us prove the converse. To this end, let $\tilde{x}\in\mathscr{H}_1$ solve problem \eqref{GSVI}. This implies that
\begin{equation}\label{GSVI1}
-A^*(I-J_\lambda^{\mathscr{B}_2}(I-\lambda f_2))A(\tilde x)\in \mathscr{B}_1(\tilde x)+f_1(\tilde x).
\end{equation}

Since $\Omega$ is nonempty by assumption, let $x^*\in\Omega$. This implies that
\begin{equation}\label{GSVI2}
-A^*(I-J_\lambda^{\mathscr{B}_2}(I-\lambda f_2))A(x^*)\in \mathscr{B}_1(\tilde x)+f_1(x^*).
\end{equation}

Since $\mathscr{B}_1+f_1$ is monotone,
\begin{align}\label{GSVI3}
\left\langle A^*(I-T)A(\tilde{x})-A^*(I-T)A(x^*), \tilde{x}-x^*\right\rangle\leq 0,
\end{align}
where $T=J_\lambda^{\mathscr{B}_2}(I-\lambda f_2)$.

That is,
\begin{align}\label{GSVI4}
\left\langle (I-T)A(\tilde{x})-(I-T)A(x^*),A(\tilde{x})-A(x^*)\right\rangle\leq 0.
\end{align}

We know that $S=A^*(I-T)A$ is $\frac{1}{2\|A\|^2}$-ism. Therefore, using \eqref{GSVI4}, we get
\begin{align}\label{GSVI5}
&\frac{1}{2\|A\|^2}\|A^*(I-T)A(\tilde{x})-A^*(I-T)A(x^*)\|^2\\&\leq\left\langle (I-T)A(\tilde{x})-(I-T)A(x^*), A(\tilde{x})-A(x^*)\right\rangle\nonumber\\&
\leq 0\nonumber.
\end{align}

Therefore,
\begin{align}\label{GSVI6}
\|A^*(I-T)A(\tilde{x})-A^*(I-T)A(x^*)\|=0.
\end{align}
That is,
\begin{align}\label{GSVI7}
A^*(I-T)A(\tilde{x})=A^*(I-T)A(x^*)=0
\end{align}
because $Ax^* = T(Ax^*)$.

This implies that
\begin{align}\label{GSVI8}
T(A\tilde{x})=A\tilde{x}+w
\end{align}
with $A^*w=0$. Combining this with the fact that $T(Ax^*)=Ax^*$, we obtain
\begin{align}\label{GSVI9}
\|T(A\tilde{x})-T(Ax^*)\|^2=\|A\tilde{x}-Ax^*\|^2+\|w\|^2.
\end{align}
Since $T$ is nonexpansive, $w=0$. This implies that
\begin{align}\label{GSVI10}
A(\tilde{x})=TA(\tilde{x}).
\end{align}
Hence
\begin{align}\label{GSVI11}
0\in\mathscr{B}_2A(\tilde{x})+f_2A(\tilde{x}).
\end{align}
Combining (\ref{GSVI1}) and (\ref{GSVI10}), we get
\begin{align}\label{GSVI12}
0\in\mathscr{B}_1(\tilde{x})+f_1(\tilde{x}).
\end{align}
Finally, using From (\ref{GSVI11}) and (\ref{GSVI12}), we now conclude that
\begin{align}\label{GSVI13}
\tilde{x}\in\Omega,
\end{align}
as asserted.
\end{proof}
\end{lemma}

In what follows, we consider a Tikhonov-type regularization technique for the class of split variational inclusions.
In particular, for each $\alpha>0$, we associate SVI with the following regularized split variational inclusion problem:
find $\tilde x\in\mathscr{H}_1$ such that
\begin{align}\label{RSVI}
0\in \mathscr{B}_1(\tilde x)+f_1(\tilde x)+A^*(I-J_\lambda^{\mathscr{B}_2}(I-\lambda f_2))A(\tilde x)+\alpha\mathscr{F}(\tilde x),
\end{align}
where $\alpha>0$ is the regularization parameter.
Our goal is to study the connections between the solutions to \eqref{SVIVI} and (\ref{RSVI}).

To this end, we first note that it follows from Assumptions $(A1)$--$(A4)$, Lemma \ref{Lem:ism:firmly} and Lemma \ref{Lem:summax} that the operator
$\mathscr{B}_1+f_1+S$ is (maximal) monotone. Since $\mathscr{F}$ is strongly monotone, $\mathscr{B}_1+f_1+S+\alpha\mathscr{F}$ is strongly monotone for each $\alpha>0$. Therefore, RSVI (\ref{RSVI}) has a unique solution for each $\alpha>0$. We denote this unique solution by $x_\alpha$.

Now we are in a position to present various analytical properties of the net $\{x_\alpha\}$.

The following results are obtained under Assumptions $(A1)$--$(A4)$ for each $\lambda \in (0,2\tilde{\tau})$.

%=============================================
\begin{lemma}\label{Lem:bound}
The net $\{x_\alpha\}_{\alpha>0}$ is bounded.
\end{lemma}

\begin{proof}
Let $x^*\in\Omega$. That is, $x^*\in\mathscr{H}_1$ satisfies
\begin{equation*}\label{SVI1}
\begin{cases}
0\in\mathscr{B}_1(x^*)+f_1(x^*)\quad\text{and}\\
y^*=Ax^*\in\mathscr{H}_2\quad\text{solves}\quad 0\in\mathscr{B}_2(y^*)+f_2(y^*).
\end{cases}
\end{equation*}
\noindent
Using (\ref{GSVI}), we get
\begin{align}\label{inq1}
0\in\mathscr{B}_1(x^*)+f_1(x^*)+A^*(I-T)A(x^*).
\end{align}
\noindent
Since $x_\alpha$ is a solution to the RSVI (\ref{RSVI}), we have
\begin{align}\label{inq2}
-\alpha\mathscr{F}(x_\alpha)\in\mathscr{B}_1(x_\alpha)+f_1(x_\alpha)+A^*(I-T)A(x_\alpha).
\end{align}
\noindent
Using the monotonicity of $\mathscr{B}_1+f_1+A^*(I-T)A$, (\ref{inq1}) and (\ref{inq2}), we find that
\begin{align*}
\left\langle\alpha\mathscr{F}(x_\alpha), x^*-x_\alpha\right\rangle\geq 0,
\end{align*}
which is equivalent to
\begin{align}\label{inq3}
\left\langle\mathscr{F}(x_\alpha), x^*-x_\alpha\right\rangle\geq 0.
\end{align}
Using the assumption that $\mathscr{F}$ is $\gamma$-strongly monotone and (\ref{inq3}), we conclude that
\begin{align}\label{inq4}
\left\langle\mathscr{F}(x^*), x^*-x_\alpha\right\rangle\geq \gamma \|x^*-x_\alpha\|^2.
\end{align}
Combining the Cauchy-Schwarz inequality and (\ref{inq4}), we obtain
\begin{align*}
\gamma \|x^*-x_\alpha\|^2&\leq\|\mathscr{F}(x^*)\|\|x^*-x_\alpha\|.
\end{align*}
This, in its turn, implies that
\begin{align}\label{inq6}
\|x^*-x_\alpha\|\leq\frac{\|\mathscr{F}(x^*)\|}{\gamma}.
\end{align}
Now, using the triangle inequality and (\ref{inq6}), we see that
\begin{align*}\label{inq62}
\|x_\alpha\|\leq\|x^*-x_\alpha\|+\|x^*\|\leq\frac{\|\mathscr{F}(x^*)\|}{\gamma}+\|x^*\|.
\end{align*}
Hence the net $\{x_\alpha\}$ is indeed bounded, as asserted.
\end{proof}
%==================================================

\begin{lemma}\label{Lem:UpperBound}
For any positive $\alpha_1$ and $\alpha_2$, we have
\begin{align}
\|x_{\alpha_1}-x_{\alpha_2}\|\leq\frac{\vert\alpha_1-\alpha_2\vert}{\alpha_1} M
\end{align}
for some positive constant $M$.
\end{lemma}

\begin{proof}
Let $x_{\alpha_1}, x_{\alpha_2}\in \{x_\alpha\}_{\alpha>0}$ be two solutions to the RSVI (\ref{RSVI}). Then we have
\begin{equation}\label{inq7}
\begin{cases}
-\alpha_1\mathscr{F}(x_{\alpha_1})\in\mathscr{B}_1(x_{\alpha_1})+f_1(x_{\alpha_1})+A^*(I-T)A(x_{\alpha_1})\\
-\alpha_2\mathscr{F}(x_{\alpha_2})\in\mathscr{B}_1(x_{\alpha_2})+f_1(x_{\alpha_2})+A^*(I-T)A(x_{\alpha_2}).
\end{cases}
\end{equation}
Since $\mathscr{B}_1+f_1+A^*(I-T)A$ is monotone, it follows from (\ref{inq7}) that
\begin{align}\label{inq8}
\left\langle\alpha_2\mathscr{F}(x_{\alpha_2})-\alpha_1\mathscr{F}(x_{\alpha_1}), x_{\alpha_1}-x_{\alpha_2}\right\rangle\geq 0.
\end{align}
Since $\mathscr{F}$ is $\gamma$-strongly monotone, we obtain
\begin{align}\label{inq9}
(\alpha_2-\alpha_1)\left\langle\mathscr{F}(x_{\alpha_2}), x_{\alpha_1}-x_{\alpha_2}\right\rangle&\geq\alpha_1\left\langle\mathscr{F}(x_{\alpha_1})-\mathscr{F}(x_{\alpha_2}), x_{\alpha_1}-x_{\alpha_2}\right\rangle\\&\geq \alpha_1\gamma\|x_{\alpha_1}-x_{\alpha_2}\|^2\nonumber.
\end{align}
Applying the Cauchy-Schwarz inequality, we get
\begin{align}\label{inq10}
\|x_{\alpha_1}-x_{\alpha_2}\|\leq\frac{\vert\alpha_1-\alpha_2\vert}{\alpha_1}\frac{\|\mathscr{F}(x_{\alpha_2})\|}{\gamma}.
\end{align}
Since $\mathscr{F}$ is Lipschitz continuous and the net $\{x_{\alpha}\}$ is bounded, so is the net $\{\mathscr{F}(x_{\alpha})\}$.

\noindent Setting $M=\sup_{\alpha >0}\Big\{\frac{\|\mathscr{F}(x_{\alpha})\|}{\gamma}\Big\}$,
we now see that
\begin{align*}
\|x_{\alpha_1}-x_{\alpha_2}\|\leq\frac{\vert\alpha_1-\alpha_2\vert}{\alpha_1}M.
\end{align*}
\end{proof}

%%===================================================
\begin{lemma}\label{Lem:WeakLim}
\begin{enumerate}
\item[(i)] $\omega(x_\alpha)\subset\Omega$, where $\omega(x_\alpha)$ denotes the set of all weak cluster points of the net $\{x_\alpha\}$.
\item[(ii)] $\lim_{\alpha\rightarrow 0^+} x_\alpha= u$, the unique solution of (\ref{SVIVI}).
\end{enumerate}
\end{lemma}

\begin{proof}
(i)
Let $\tilde{x}\in\omega(x_\alpha)$. Then there exists a subsequence $\{x_{\alpha_n}\}$ of  $\{x_{\alpha}\}$ such that $\alpha_n\rightarrow 0^+$ and $x_{\alpha_n}\rightharpoonup \tilde{x}$ as $n\rightarrow\infty$.

\noindent
Let $(y, z)\in gra(\mathscr{B}_1+f_1+A^*(I-T)A)$. Then we have
\begin{align}\label{inq12}
z-f_1(y)-A^*(I-T)A(y)\in\mathscr{B}_1(y).
\end{align}
On the other hand, using the definition of $\{x_\alpha\}$, we also have
\begin{align}\label{inq13}
-\alpha\mathscr{F}(x_\alpha)-f_1(x_\alpha)-A^*(I-T)A(x_\alpha)\in\mathscr{B}_1(x_\alpha).
\end{align}
Using the monotonicity of $\mathscr{B}_1$, (\ref{inq12}) and (\ref{inq13}), we obtain,
{\small
\begin{align}\label{inq14}
\left\langle z-f_1(y)-A^*(I-T)A(y)+\alpha\mathscr{F}(x_\alpha)+f_1(x_\alpha)+A^*(I-T)A(x_\alpha), y-x_\alpha\right\rangle\geq 0.
\end{align}}

\noindent
Again using the fact that $A^*(I-T)A$ is monotone, as well as (\ref{inq14}), we now get
\begin{align*}
\left\langle z-f_1(y)+\alpha\mathscr{F}(x_\alpha)+f_1(x_\alpha), y-x_\alpha\right\rangle\geq 0.
\end{align*}
That is,
\begin{align}\label{inq16}
\left\langle z+\alpha\mathscr{F}(x_\alpha), y-x_\alpha\right\rangle\geq \left\langle f_1(y)-f_1(x_\alpha), y-x_\alpha\right\rangle.
\end{align}
Since $f_1$ is monotone, it follows from (\ref{inq16}) that
\begin{align*}
\left\langle z+\alpha\mathscr{F}(x_\alpha), y-x_\alpha\right\rangle\geq 0.
\end{align*}
In particular, we have
\begin{align*}
\left\langle z+\alpha_n\mathscr{F}(x_{\alpha_n}), y-x_{\alpha_n}\right\rangle\geq 0,
\end{align*}
which is equivalent to
\begin{align}\label{inq18}
\left\langle z, y-x_{\alpha_n}\right\rangle\geq\alpha_n\left\langle \mathscr{F}(x_{\alpha_n}), x_{\alpha_n}-y\right\rangle.
\end{align}
Since $\mathscr{F}$ is Lipschitz continuous and $\{x_{\alpha_n}\}$ is bounded, the real sequence $\{\left\langle\mathscr{F}(x_{\alpha_n}), x_{\alpha_n}-y\right\rangle\}$ is bounded too.
Passing to the limit as $n\rightarrow\infty$ in (\ref{inq18}), we obtain
\begin{align}\label{inq19}
\left\langle z, y-\tilde{x}\right\rangle\geq 0,\quad\forall (y, z)\in gra(\mathscr{B}_1+f_1+A^*(I-T)A).
\end{align}
Since $\mathscr{B}_1+f_1+A^*(I-T)A$ is maximal monotone, it follows from \eqref{inq19} that
\begin{align}\label{inq20}
0\in (\mathscr{B}_1+f_1+A^*(I-T)A)(\tilde{x}).
\end{align}
This implies that
\begin{align}\label{inq21}
-f_1(\tilde{x})-A^*(I-T)A(\tilde{x})\in\mathscr{B}_1(\tilde{x}).
\end{align}
Consider now an arbitrary point $x^*\in\Omega$. Then we have $0\in\mathscr{B}_1(x^*)+f_1(x^*)+A^*(I-T)A(x^*)$ and therefore,
\begin{align}\label{inq22}
-f_1(x^*)-A^*(I-T)A(x^*)\in\mathscr{B}_1(x^*).
\end{align}
Again using the monotonicity of $\mathscr{B}_1$, it follows from (\ref{inq21}) and (\ref{inq22}) that
{\small
\begin{align}\label{inq23}
\left\langle f_1(x^*)+A^*(I-T)A(x^*)-f_1(\tilde{x})-A^*(I-T)A(\tilde{x}), \tilde{x}-x^*\right\rangle\geq 0.
\end{align}}

\noindent
Using the monotonicity of $f_1$ and (\ref{inq23}), we obtain
\begin{align*}
&\left\langle A^*(I-T)A(x^*)-A^*(I-T)A(\tilde{x}), \tilde{x}-x^*\right\rangle\\&
\geq\left\langle f_1(\tilde{x})-f_1(x^*), \tilde{x}-x^*\right\rangle\geq 0\nonumber.
\end{align*}
That is,
\begin{align}\label{inq25}
\left\langle A^*(I-T)A(\tilde{x})-A^*(I-T)A(x^*), \tilde{x}-x^*\right\rangle\leq 0.
\end{align}
Since $A^*(I-T)A$ is $\frac{1}{2\|A\|^2}$-ism, we have
\begin{align}\label{inq26}
&\frac{1}{2\|A\|^2}\|A^*(I-T)A(\tilde{x})-A^*(I-T)A(x^*)\|^2\\&\leq\left\langle A^*(I-T)A(\tilde{x})-A^*(I-T)A(x^*), \tilde{x}-x^*\right\rangle\nonumber\\&
\leq 0\nonumber.
\end{align}
Therefore,
\begin{align}\label{inq27}
\|A^*(I-T)A(\tilde{x})-A^*(I-T)A(x^*)\|=0.
\end{align}
That is,
\begin{align}\label{inq28}
A^*(I-T)A(\tilde{x})=A^*(I-T)A(x^*)=0.
\end{align}
This implies that
\begin{align}\label{ineq28:1}
T(A\tilde{x})=A\tilde{x}+w
\end{align}
with $A^*w=0$. Combining this with the fact that $T(Ax^*)=Ax^*$, we obtain
\begin{align}\label{ineq28:2}
\|T(A\tilde{x})-T(Ax^*)\|^2=\|A\tilde{x}-Ax^*\|^2+\|w\|^2.
\end{align}
Since $T$ is nonexpansive, it follows that $w=0$.
This implies that
\begin{align}\label{inq29}
A(\tilde{x})=TA(\tilde{x}).
\end{align}
Hence
\begin{align}\label{inq30}
0\in\mathscr{B}_2A(\tilde{x})+f_2A(\tilde{x}).
\end{align}
Combining (\ref{inq20}) and (\ref{inq29}), we get
\begin{align}\label{inq31}
0\in\mathscr{B}_1(\tilde{x})+f_1(\tilde{x}).
\end{align}
From (\ref{inq30}) and (\ref{inq31}), it follows that
\begin{align}\label{inq32}
\tilde{x}\in\Omega.
\end{align}
\noindent
Since $\tilde{x}$ is arbitrary, we conclude that
\begin{align}\label{inq322}
\omega(x_\alpha)\subset\Omega,
\end{align}
as asserted.

(ii)~ Let $\tilde{x}\in w(x_\alpha)$. Then there exists a subsequence $\{x_{\alpha_n}\}$ of $\{x_\alpha\}$ such that $\alpha_n\rightarrow 0^+$ and $x_{\alpha_n}\rightharpoonup \tilde{x}$ as $n\rightarrow\infty$. We already know that $\tilde{x}\in\Omega$.
Using the strong monotonicity of $\mathscr{F}$ and (\ref{inq3}), we obtain
\begin{align}\label{inq33}
\left\langle\mathscr{F}(x^*), x^*-x_\alpha\right\rangle\geq\gamma\|x^*-x_\alpha\|^2 \; \quad\forall x^*\in\Omega.
\end{align}
This implies that
\begin{align}\label{inq34}
\left\langle\mathscr{F}(x^*), x^*-x_\alpha\right\rangle\geq 0 \; \quad\forall x^*\in\Omega.
\end{align}
Setting $\alpha = \alpha_n$ and passing to the limit as $n\rightarrow\infty$ in (\ref{inq34}), we get
\begin{align}\label{inq35}
\left\langle\mathscr{F}(x^*), x^*-\tilde{x}\right\rangle\geq 0 \; \quad\forall x^*\in\Omega.
\end{align}
\noindent
Using Lemma \ref{Lem:Minty}, we get
\begin{align}\label{inq36}
\left\langle\mathscr{F}(\tilde{x}), x^*-\tilde{x}\right\rangle\geq 0 \; \quad\forall x^*\in\Omega.
\end{align}
Since the VIP (\ref{SVIVI}) has a unique solution, we have $\tilde{x}=u$. Therefore, $\omega(x_\alpha)=\{u\}$, that is, the whole net $\{x_\alpha\}$ converges weakly to $u$.
Substituting $x^*=u$ in (\ref{inq33}), we get
\begin{align}\label{inq37}
\left\langle\mathscr{F}(u), u-x_\alpha\right\rangle\geq \gamma\|u-x_\alpha\|^2.
\end{align}
This yields
\begin{align}\label{inq38}
\|u-x_\alpha\|^2\leq\left\langle\mathscr{F}(u), u-x_\alpha\right\rangle/\gamma.
\end{align}
Since $\{x_\alpha\}$ converges weakly to $u$,
passing to the limit $\alpha\rightarrow 0^+$ in (\ref{inq38}), we obtain
\begin{align}\label{inq39}
\|u-x_\alpha\|\rightarrow 0.
\end{align}
That is,
\begin{align*}
x_\alpha\rightarrow u\text{ as } \alpha\rightarrow 0^+,
\end{align*}
as asserted.
\end{proof}
%%==========================================================================================================================
\section{Our Algorithm}
\label{Sec:4}
\noindent
In this section we present an iterative algorithm for solving the RSVI (\ref{RSVI}) and present its convergence analysis.
\begin{theorem}
Suppose that Assumptions $(A1)$--$(A4)$ hold. Let $z_1\in \mathscr{H}_1$ be arbitrary. Compute
\begin{align}\label{Algo1}
z_{n+1}=J_{\lambda_n}^{\mathscr{B}_1}(z_n-\lambda_n f_1(z_n)-\lambda_n A^*(I-J_{\lambda_n}^{\mathscr{B}_2}(I-\lambda_n f_2))A(z_n)-\lambda_n\alpha_n\mathscr{F}(z_n)),
\end{align}
where the parameters satisfy the following conditions:
\begin{itemize}
\item [(B1)] $\alpha_n\in (0, 1)$, $\lim\limits_{n\to \infty}\alpha_n = 0$ and $\sum\limits_{n=1}^{\infty}\alpha_n = \infty$;
\item[(B2)] $\lim\limits_{n\to \infty}\frac{\vert\alpha_{n+1}-\alpha_n\vert}{\alpha_n^2}=0$;
\item[(B3)] $0<c<\lambda_n<\frac{1}{\rho}min \Big\{\tilde{\tau},\frac{1}{\|A\|^2}\Big\}$ for each $n\in\mathbb{N}$ with $\tilde{\tau}=\min\{\tau_1,\tau_2\}$ and $\rho>2$.
\end{itemize}
Then the sequence $\{z_n\}$ generated by the iterative algorithm \eqref{Algo1} converges strongly to a solution $u\in\Omega$ of the problem SVI \eqref{SVI}, which solves uniquely problem (\ref{SVIVI}).
\end{theorem}

\begin{proof}
Let $x_{\alpha_n}$ be the solution to the RSVI with $\alpha$ replaced by $\alpha_n$ and $T:=A^*(I-J_{\lambda}^{\mathscr{B}_2}(I-\lambda f_2))A$.
We have
\begin{align}\label{ineq40}
0\in \mathscr{B}_1x_{\alpha_n}+ f_1(x_{\alpha_n})+ A^*(I-J_{\tau}^{\mathscr{B}_2}(I-\tau f_2))Ax_{\alpha_n}+\alpha_n\mathscr{F}x_{\alpha_n}.
\end{align}
Using (\ref{Algo1}), we obtain
\begin{align*}
z_n-\lambda_n f_1(z_n)-\lambda_n T(z_n)-\lambda_n\alpha_n\mathscr{F}(z_n)\in z_{n+1}+\lambda_n\mathscr{B}_1(z_{n+1}),
\end{align*}
or equivalently,
\begin{align}\label{ineq42}
z_n-z_{n+1}-\lambda_n f_1(z_n)-\lambda_n T(z_n)-\lambda_n\alpha_n\mathscr{F}(z_n)\in\lambda_n\mathscr{B}_1(z_{n+1}).
\end{align}
Similarly, using (\ref{ineq40}), we get
\begin{align}\label{ineq43}
-\lambda_n f_1(x_{\alpha_n})-\lambda_n T(x_{\alpha_n})-\lambda_n\alpha_n\mathscr{F}(x_{\alpha_n})\in\lambda_n\mathscr{B}_1(x_{\alpha_n}).
\end{align}
Since $\mathscr{B}_1$ is monotone, combining (\ref{ineq42}) and (\ref{ineq43}), we find that
{\tiny \begin{align*}
\left\langle z_n-z_{n+1}-\lambda_n f_1(z_n)-\lambda_n T(z_n)-\lambda_n\alpha_n\mathscr{F}(z_n)+\lambda_n f_1(x_{\alpha_n})+\lambda_n T(x_{\alpha_n})+\lambda_n\alpha_n\mathscr{F}(x_{\alpha_n}), z_{n+1}-x_{\alpha_n}\right\rangle\geq 0.
\end{align*}}

\noindent
This implies that
\begin{align}\label{ineq45}
\left\langle z_n-z_{n+1},z_{n+1}-x_{\alpha_n}\right\rangle\geq & \lambda_n\left\langle f_1(z_n)-f_1(x_{\alpha_n}),z_{n+1}-x_{\alpha_n}\right\rangle\\&+\lambda_n\left\langle T(z_n)-T(x_{\alpha_n}), z_{n+1}-x_{\alpha_n}\right\rangle\nonumber\\&+\lambda_n\alpha_n\left\langle   \mathscr{F}(z_n)- \mathscr{F}(x_{\alpha_n}),z_{n+1}-x_{\alpha_n}\right\rangle\nonumber.
\end{align}
Next, we observe that
\begin{align}\label{ineq46}
\left\langle z_n-z_{n+1},z_{n+1}-x_{\alpha_n}\right\rangle=\frac{1}{2}[\|z_n-x_{\alpha_n}\|^2-\|z_n-z_{n+1}\|^2-\|z_{n+1}-x_{\alpha_n}\|^2].
\end{align}
Since $f_1$ is $\tilde{\tau}$-ism, we have
\begin{align}\label{ineq47}
\left\langle f_1(z_n)-f_1(x_{\alpha_n}),z_{n+1}-x_{\alpha_n}\right\rangle&=\left\langle f_1(z_n)-f_1(x_{\alpha_n}),z_{n+1}-z_n\right\rangle\\&+\left\langle f_1(z_n)-f_1(x_{\alpha_n}),z_n-x_{\alpha_n}\right\rangle\nonumber\\&
\geq \left\langle f_1(z_n)-f_1(x_{\alpha_n}),z_{n+1}-z_n\right\rangle\nonumber\\&
+\tilde{\tau}\|f_1(z_n)-f_1(x_{\alpha_n})\|^2.\nonumber
\end{align}
Using the fact $T$ is $\frac{1}{2\|A\|^2}$-ism, we get
\begin{align}\label{ineq48}
\left\langle T(z_n)-T(x_{\alpha_n}),z_{n+1}-x_{\alpha_n}\right\rangle&=\left\langle T(z_n)-T(x_{\alpha_n}),z_{n+1}-z_n\right\rangle\\&+\left\langle T(z_n)-T(x_{\alpha_n}),z_n-x_{\alpha_n}\right\rangle\nonumber\\&
\geq \left\langle T(z_n)-T(x_{\alpha_n}),z_{n+1}-z_n\right\rangle\nonumber\\&
+\frac{1}{2\|A\|^2}\|T(z_n)-T(x_{\alpha_n})\|^2.\nonumber
\end{align}
Since $\mathscr{F}$ is $\gamma$-strongly monotone, we have
\begin{align}\label{ineq49}
\left\langle \mathscr{F}(z_n)-\mathscr{F}(x_{\alpha_n}),z_{n+1}-x_{\alpha_n}\right\rangle&=\left\langle \mathscr{F}(z_n)-\mathscr{F}(x_{\alpha_n}),z_{n+1}-z_n\right\rangle\\&+\left\langle \mathscr{F}(z_n)-\mathscr{F}(x_{\alpha_n}),z_n-x_{\alpha_n}\right\rangle\nonumber\\&
\geq \left\langle \mathscr{F}(z_n)-\mathscr{F}(x_{\alpha_n}),z_{n+1}-z_n\right\rangle+\gamma\|z_n-x_{\alpha_n}\|^2.\nonumber
\end{align}
Using (\ref{ineq46})--(\ref{ineq49}) in (\ref{ineq45}), we obtain
\begin{align}\label{ineq50}
\frac{1}{2}[\|z_n-x_{\alpha_n}\|^2- & \|z_n-z_{n+1}\|^2-\|z_{n+1}-x_{\alpha_n}\|^2]\\&
\geq \lambda_n\left\langle f_1(z_n)-f_1(x_{\alpha_n}),z_{n+1}-z_n\right\rangle+\lambda_n\tilde{\tau}\|f_1(z_n)-f_1(x_{\alpha_n})\|^2\nonumber\\&
+\lambda_n\left\langle T(z_n)-T(x_{\alpha_n}),z_{n+1}-z_n\right\rangle+\frac{\lambda_n}{2\|A\|^2}\|T(z_n)-T(x_{\alpha_n})\|^2\nonumber\\&
+\lambda_n\alpha_n\left\langle \mathscr{F}(z_n)-\mathscr{F}(x_{\alpha_n}),z_{n+1}-z_n\right\rangle+\gamma\lambda_n\alpha_n\|z_n-x_{\alpha_n}\|^2.\nonumber
\end{align}
This implies that
\begin{align}\label{ineq51}
\|z_{n+1}-x_{\alpha_n}\|^2 &\leq  \|z_n-x_{\alpha_n}\|^2-\|z_n-z_{n+1}\|^2\\&
-2\lambda_n\left\langle f_1(z_n)-f_1(x_{\alpha_n}),z_{n+1}-z_n\right\rangle\nonumber\\&-2\lambda_n\tilde{\tau}\|f_1(z_n)-f_1(x_{\alpha_n})\|^2-2\lambda_n\left\langle T(z_n)-T(x_{\alpha_n}),z_{n+1}-z_n\right\rangle\nonumber\\&-\frac{\lambda_n}{\|A\|^2}\|T(z_n)-T(x_{\alpha_n})\|^2
-2\lambda_n\alpha_n\left\langle \mathscr{F}(z_n)-\mathscr{F}(x_{\alpha_n}),z_{n+1}-z_n\right\rangle\nonumber\\&-2\gamma\lambda_n\alpha_n\|z_n-x_{\alpha_n}\|^2.\nonumber
\end{align}
Using Young's inequality and (\ref{ineq51}), we have
\begin{align*}
\|z_{n+1}-x_{\alpha_n}\|^2 &\leq  (1-2\gamma\lambda_n\alpha_n)\|z_n-x_{\alpha_n}\|^2+\lambda_n\epsilon_1\|f_1(z_n)-f_1(x_{\alpha_n})\|^2\\&
+\frac{\lambda_n}{\epsilon_1}\|z_n-z_{n+1}\|^2-2\lambda_n\tilde{\tau}\|f_1(z_n)-f_1(x_{\alpha_n})\|^2\\&
+\lambda_n\epsilon_2\|T(z_n)-T(x_{\alpha_n})\|^2
+\frac{\lambda_n}{\epsilon_2}\|z_n-z_{n+1}\|^2\\& 
-\frac{\lambda_n}{\|A\|^2}\|T(z_n)-T(x_{\alpha_n})\|^2+\lambda_n\alpha_n L\epsilon_3\|z_n-x_{\alpha_n}\|^2\\&
+\frac{\lambda_n\alpha_n L}{\epsilon_3}\|z_n-z_{n+1}\|^2\nonumber-\|z_n-z_{n+1}\|^2,
\end{align*}
where $\epsilon_1,\epsilon_2,\epsilon_3>0.$ That is,
\begin{align}\label{ineq53}
\|z_{n+1}-x_{\alpha_n}\|^2 &\leq  (1-2\gamma\lambda_n\alpha_n+\lambda_n\alpha_n L\epsilon_3)\|z_n-x_{\alpha_n}\|^2\\&
+\lambda_n(\epsilon_1-2\tilde{\tau})\|f_1(z_n)-f_1(x_{\alpha_n})\|^2\nonumber\\&
+\lambda_n\Big(\epsilon_2-\frac{1}{\|A\|^2}\Big)\|T(z_n)-T(x_{\alpha_n})\|^2\nonumber\\&
-\Big(1-\frac{\lambda_n}{\epsilon_1}-\frac{\lambda_n}{\epsilon_2}-\frac{\lambda_n\alpha_n L}{\epsilon_3}\Big)\|z_n-z_{n+1}\|^2\nonumber.
\end{align}
Choose $\epsilon_1=\lambda_n\rho=\epsilon_2$, $\epsilon_3=\gamma/L$.
Then it follows that there exists $N_0\in\mathbb{N}$ such that
\begin{align}\label{ineq54}
\|z_{n+1}-x_{\alpha_n}\|^2 &\leq  (1-\gamma\lambda_n\alpha_n)\|z_n-x_{\alpha_n}\|^2-\frac{1}{2}\Big(1-\frac{2}{\rho}\Big)\|z_n-z_{n+1}\|^2 \; \quad\forall n\geq N_0.
\end{align}
Set $\kappa_n=\gamma\lambda_n\alpha_n$. Then it follows from (\ref{ineq54}) that
\begin{align}\label{ineq55}
\|z_{n+1}-x_{\alpha_n}\|^2 &\leq  (1-\kappa_n)\|z_n-x_{\alpha_n}\|^2-\frac{1}{2}\Big(1-\frac{2}{\rho}\Big)\|z_n-z_{n+1}\|^2 \; \quad\forall n\geq N_0.
\end{align}
Now, once again using Young's inequality and Lemma \ref{Lem:UpperBound}, we obtain
\begin{align}\label{ineq56}
\|z_{n+1}-x_{\alpha_n}\|^2 &=\|z_{n+1}-x_{\alpha_{n+1}}+x_{\alpha_{n+1}}-x_{\alpha_n}\|^2\\&
=\|z_{n+1}-x_{\alpha_{n+1}}\|^2+2\left\langle z_{n+1}-x_{\alpha_{n+1}},x_{\alpha_{n+1}}-x_{\alpha_n}\right\rangle\nonumber\\&
+\|x_{\alpha_{n+1}}-x_{\alpha_n}\|^2\nonumber\\&
\geq \|z_{n+1}-x_{\alpha_{n+1}}\|^2-\epsilon_4 \|z_{n+1}-x_{\alpha_{n+1}}\|^2\nonumber\\&
-\frac{\|x_{\alpha_{n+1}}-x_{\alpha_n}\|^2}{\epsilon_4}+\|x_{\alpha_{n+1}}-x_{\alpha_n}\|^2\nonumber\\&
= (1-\epsilon_4)\|z_{n+1}-x_{\alpha_{n+1}}\|^2-\frac{(1-\epsilon_4)}{\epsilon_4}\|x_{\alpha_{n+1}}-x_{\alpha_n}\|^2\nonumber\\&
\geq (1-\epsilon_4)\|z_{n+1}-x_{\alpha_{n+1}}\|^2-\frac{(1-\epsilon_4)}{\epsilon_4}\frac{(\alpha_{n+1}-\alpha_n)^2}{\alpha_n^2}M^2\nonumber,
\end{align}
where $\epsilon_4>0$. From (\ref{ineq55}) and (\ref{ineq56}) it follows that
\begin{align}\label{ineq57}
(1-\epsilon_4)\|z_{n+1}-x_{\alpha_{n+1}}\|^2\leq  (1-\kappa_n)&\|z_n-x_{\alpha_n}\|^2-\frac{1}{2}\Big(1-\frac{2}{\rho}\Big)\|z_n-z_{n+1}\|^2\\&
+\frac{(1-\epsilon_4)}{\epsilon_4}\frac{(\alpha_{n+1}-\alpha_n)^2}{\alpha_n^2}M^2 \; \quad\forall n\geq N_0\nonumber.
\end{align}
Choose $\epsilon_4=\frac{\kappa_n}{2}$ and $a_n=\frac{\epsilon_4}{1-\epsilon_4}$. Then it follows from (\ref{ineq57}) that
\begin{align}\label{ineq58}
\|z_{n+1}-x_{\alpha_{n+1}}\|^2\leq(1-a_n)&\|z_n-x_{\alpha_n}\|^2-\frac{1}{2(1-\epsilon_4)}\Big(1-\frac{2}{\rho}\Big)\|z_n-z_{n+1}\|^2\\&
+\frac{1}{\epsilon_4}\frac{(\alpha_{n+1}-\alpha_n)^2}{\alpha_n^2}M^2 \; \quad\forall n\geq N_0\nonumber.
\end{align}
This implies that
\begin{align}\label{ineq59}
\|z_{n+1}-x_{\alpha_{n+1}}\|^2\leq(1-a_n)&\|z_n-x_{\alpha_n}\|^2-\frac{1}{2(1-\epsilon_4)}\Big(1-\frac{2}{\rho}\Big)\|z_n-z_{n+1}\|^2\\&
+a_n\frac{(1-\epsilon_4)}{\epsilon_4^2}\frac{(\alpha_{n+1}-\alpha_n)^2}{\alpha_n^2}M^2 \; \quad\forall n\geq N_0\nonumber.
\end{align}
Let $\Gamma_n=\|z_n-x_{\alpha_n}\|^2$ and $b_n=\frac{(1-\epsilon_4)}{\epsilon_4^2}\frac{(\alpha_{n+1}-\alpha_n)^2}{\alpha_n^2}M^2$. Then inequality (\ref{ineq59}) is reduced to the following form:
\begin{align}\label{ineq60}
\Gamma_{n+1}\leq(1-a_n)\Gamma_n+a_nb_n \;
\quad\forall n\geq N_0.
\end{align}

Since $a_n\in (0,1)$, $\sum_{n=1}^{\infty}a_n=+\infty$ and $\lim_{n\rightarrow\infty}b_n = 0$, we can invoke Lemma \ref{Lem:Lastconvergence} to conclude that
$\lim_{n\rightarrow\infty}\Gamma_n=0$. Thus $z_n- x_{\alpha_n}\to 0$ as $n\to \infty$. Combining this with Lemma \ref{Lem:WeakLim} (ii),
we find that $z_n\rightarrow u$ as $n\rightarrow\infty$, as asserted. This completes the proof.
\end{proof}

%==================================================================================================================================
%Using Lemma \ref{Lem:bound}, it is clear that $\{b_n\}$ is bounded. We next show that for every subsequence $\{\Gamma_{n_k}\}$ of $\{\Gamma_n\}$ that satisfies %$\liminf_{k\rightarrow\infty}(\Gamma_{n_k+1}-\Gamma_{n_k})\geq 0$ implies that $\limsup_{k\rightarrow\infty}\Gamma_{n_k}\leq 0$. Indeed, let subsequence $\{\Gamma_{n_k}\}$ be a %subsequence of $\{\Gamma_n\}$ that satisfies $\liminf_{k\rightarrow\infty}(\Gamma_{n_k}-\Gamma_{n_k})\geq 0$.
%It is clear that $a_n\in (0,1)$, $\sum_{n=1}^{\infty}a_n=+\infty$ and  $\limsup_{k\rightarrow\infty}b_{n_k} = 0$. Therefore, by Lemma \ref{Lem:Lastconvergence}, we conclude that %$\lim_{n\rightarrow\infty}\Gamma_n=0$.
%==========================================================================================

\begin{remark}\label{remark:Weakalgo}
If we put $\alpha_n = 0$ in \eqref{Algo1} for all $n\in\mathbb{N}$, then we obtain the following iterative scheme: For a given $z_1\in\mathscr{H}$, compute
\begin{equation}\label{Algo2}
z_{n+1}=J_{\lambda_n}^{\mathscr{B}_1}(z_n-\lambda_n f_1(z_n)-\lambda_n A^*(I-T)A(z_n)),
\end{equation}
where $T=J_{\lambda_n}^{\mathscr{B}_2}(I-\lambda_n f_2)$. Note that \eqref{Algo2} is different from the iterative method \eqref{algo:moudafi} which was proposed by Moudafi for solving the SVI (\ref{SVI}). Thus \eqref{Algo2} provides a novel scheme for solving the SVI \eqref{SVI}.
Next, we show that the sequence generated by \eqref{Algo2} converges weakly to a solution of \eqref{SVI} under certain assumptions.
\end{remark}

\begin{corollary}\label{Thm:Weakalgo}
Under assumptions (A1), (A2), (A4) and (B3), the sequence $\{z_n\}$ generated by \eqref{Algo2} converges weakly to a solution of the SVI \eqref{SVI}.
\end{corollary}
\begin{proof}
Let $u\in \Omega$.
Taking $\alpha_n = 0$ for all natural numbers $n$ in the analysis \eqref{ineq40}--\eqref{ineq55}, we find that
\begin{eqnarray}
\|z_{n+1} - u\|^2&\leq& \|z_n - u\|^2 - \frac{1}{2}\Big(1-\frac{2}{\rho}\Big)\|z_n - z_{n+1}\|^2\label{61}\\
&\leq& \|z_n - u\|^2.\label{62}
\end{eqnarray}
Thus the sequence $\{\|z_n- u\|\}$ is monotonically decreasing. Therefore $\{\|z_n - u\|\}$ is bounded and hence $\{z_n\}$ is bounded.
Also, the limit of the sequence $\{\|z_n- u\|\}$ exists. Since $\rho > 2$, it follows from \eqref{61} that
\begin{equation*}\label{63}
\lim_{n\to \infty}\|z_{n+1} - z_n\| = 0.
\end{equation*}
Consider a weakly convergent subsequence of the sequence $\{z_n\}$, say $z_{n_k}\rightharpoonup \bar{x}$.
We claim that $\bar{x}\in \Omega$. Indeed, let $(x,y)\in gra(\mathscr{B}_1 + f_1 + A^*(I-T)A)$ be arbitrary. Then we have
\begin{equation}\label{64}
y- f_1x - A^*(I-T)Ax\in \mathscr{B}_1x.
\end{equation} 
It follows from \eqref{Algo2} that
\begin{equation}\label{65}
 z_{n_k} - z_{n_k+1}- \lambda_{n_k}f_1z_{n_k} - \lambda_{n_k}A^*(I-T)Az_{n_k}\in \lambda_{n_k}\mathscr{B}_1z_{n_k+1}.
\end{equation}
Using \eqref{64} and \eqref{65}, we infer that 
\begin{equation}\label{66}
\langle \frac{z_{n_k} - z_{{n_k}+1}}{\lambda_{n_k}}- f_1z_{n_k} - A^*(I-T)Az_{n_k} - (y- f_1x - A^*(I-T)Ax), z_{{n_k}+1} - x\rangle \geq 0.
\end{equation}
Taking the limit as $k\to \infty$ and using \eqref{66}, we find that
\begin{eqnarray}\label{67}
\langle y, x - \bar{x}\rangle &\geq& \langle f_1x + A^*(I-T)Ax - f_1\bar{x} - A^*(I-T)A\bar{x}, x - \bar{x}\rangle \nonumber\\
&\geq& 0.
\end{eqnarray}
Since $\mathscr{B}_1 + f_1 + A^*(I-T)A$ is maximal monotone, it follows from \eqref{67} that $0\in (\mathscr{B}_1 + f_1 + A^*(I-T)A)\bar{x}$ and hence $\bar{x}\in \Omega$, as claimed. Using Lemma \ref{Xu:lem}, we now conclude that the whole sequence $\{z_n\}$ converges weakly to a point in $\Omega$, as asserted. This completes the proof.
\end{proof}

%==========================================================================================================================
%\begin{remark}
%If we put $f_1=f_2 = 0$ in \eqref{Algo2}, then we obtain the following iterative scheme: For given $x_1\in\mathscr{H}$, compute
%\begin{equation}\label{Algo2}
%x_{n+1}=J_{\lambda_n}^{\mathscr{B}_1}(x_n-\lambda_n A^*(I-J_{\lambda_n}^{\mathscr{B}_2})Ax_n).
%\end{equation}
%\end{remark}

%%==========================================================================================================================
\section{Applications}
\label{Sec:5}
Let $\mathscr{H}$ be a real Hilbert space and $f:\mathscr{H}\rightarrow\mathbb{R}\cup\{+\infty\}$ be a lower semi-continuous convex function. Recall that the subdifferential of $f$ at a point $x\in\mathscr{H}$, denoted by $\partial f(x)$, is defined by
\begin{align}
\partial f(x)=\{u\in H: f(y)\geq f(x)+\left\langle u, y-x\right\rangle \; \quad \forall y\in\mathscr{H}\}.
\end{align}

Let $C$ be a nonempty, closed and convex subset of $\mathscr{H}$. Then the indicator function of $C$ is defined by
\begin{equation}
i_C(x) :=
\begin{cases}
0, &\text{  if }x\in C\\
\infty,&\text{ otherwise}.
\end{cases}
\end{equation}

The normal cone to a set $C$ at a point $x$ in $C$, denoted by $\mathscr{N}_C(x)$, is defined by
\begin{align}
\mathscr{N}_C(x) :=
\begin{cases}
\{u\in \mathscr{H}: \left\langle u, y-x\right\rangle\leq 0 \; \quad \forall y\in C\}, &\quad\text{if } x\in C\\
\emptyset, &\quad \text{otherwise}.
\end{cases}
\end{align}

It is well known that if $x \in int(C)\neq\emptyset$, where $int(C$ denotes the set of all interior points of $C$, then $\mathscr{N}_C(x)=\left\lbrace 0\right\rbrace$. 
The resolvent of $\mathscr{N}_C$ is $P_C$, the metric projection of $\mathscr{H}$ onto $C$, defined by the following formula:
\begin{align*}
P_C(x) := \arg\min_{y\in C} \|x-y\| \; \quad\forall x\in\mathscr{H}.
\end{align*}

Also, it is evident that $\partial i_C(x)=\mathscr{N}_C(x)$, where $x$ is in $\mathscr{H}$, is a nonempty, closed and convex subset of $\mathscr{H}$. Furthermore, $\partial i_C$ is a maximal monotone set-valued mapping from $\mathscr{H}$ to $2^{\mathscr{H}}$ and $P_C  = (I+\lambda \partial i_C)^{-1}$, $\lambda > 0$. More details on this topic can be found in \cite{HHBA2011}.

\subsection{Split convex minimization problem}\label{Sec:5.1}
Let $\mathscr{H}_1$ and $\mathscr{H}_2$ be real Hilbert spaces and $E:\mathscr{H}_1\rightarrow\mathbb{R}\cup \{+\infty\} $ and $G:\mathscr{H}_2\rightarrow\mathbb{R}\cup \{+\infty\}$ be two functions that can be decomposed as sum of two functions, i.e.,
\begin{equation*}
	E(x)\equiv e_1(x) + e_2(x)
\end{equation*}
and
\begin{equation*}
	G(y)\equiv g_1(y) + g_2(y),
\end{equation*}
in which $e_1: \mathscr{H}_1\rightarrow\mathbb{R}\cup \{+\infty\}$ and $g_1: \mathscr{H}_2\rightarrow\mathbb{R}\cup \{+\infty\}$ are proper lower semi-continuous convex functions, and $e_2: \mathscr{H}_1\rightarrow\mathbb{R}$ and $g_2: \mathscr{H}_2\rightarrow\mathbb{R}$ and convex and diffentiable functions. Let $A:\mathscr{H}_1\rightarrow \mathscr{H}_2$ be a bounded linear operator. Then the {\em split convex minimization problem} (SCMP) is defined as follows: find $x^*\in \mathscr{H}_1$ such that
\begin{equation}\label{Split:MIN}
\begin{cases}
x^*=\arg\min_{x\in \mathscr{H}_1} E(x)\text{ and }\\
y^*=Ax^* \text{ such that }y^*= \arg\min_{y\in \mathscr{H}_2} G(y).
\end{cases}
\end{equation}

If we put $B_1=\partial e_1$, $B_2=\partial g_1$, $f_1=\nabla e_2$, and $f_2=\nabla g_2$, then the SVI \eqref{SVI} reduces to the SCMP \eqref{Split:MIN} (see \cite{RT,Tai}).

Therefore, we deduce the following strong convergence theorem for SCMP \eqref{Split:MIN} from Theorem \eqref{Algo1} under some additional assumptions on the parameters and mappings involved.

Assume that the solution set $\Omega$ of the problem \eqref{Split:MIN} is nonempty.

\begin{theorem}
Suppose the above-mentioned data are given. Assume that the gradients $\nabla e_2$ and $\nabla g_2$ are Lipschitz continuous with constants $L_1$ and $L_2$, respectively. For a given $z_1\in \mathscr{H}_1$, compute
\begin{align}\label{Algo:SMP}
z_{n+1}=J^{\partial e_1}_{\lambda_n}(z_n-\lambda_n \nabla e_2(z_n)-\lambda_n A^*(I-J^{\partial g_1}_{\lambda_n}(I-\lambda_n \nabla g_2))A(z_n)-\lambda_n\alpha_n\mathscr{F}(z_n)),
\end{align}
where the parameters satisfy the following conditions:
\begin{itemize}
\item [(C1)] $\alpha_n\in (0, 1)$, $\lim\limits_{n\to \infty}\alpha_n = 0$ and $\sum\limits_{n=1}^{\infty}\alpha_n = \infty$;
\item[(C2)] $\lim\limits_{n\to \infty}\frac{\vert\alpha_{n+1}-\alpha_n\vert}{\alpha_n^2}=0$;
\item[(C3)] $0<c<\lambda_n<\frac{1}{\rho}min \Big\{\tilde{\tau},\frac{1}{\|A\|^2}\Big\}$ for each $n\in\mathbb{N}$, where $\tilde{\tau}=\min\{\frac{1}{L_1},\frac{1}{L_2}\}$ and $\rho>2$.
\end{itemize}
Then the sequence $\{z_n\}$ generated by the iterative algorithm \eqref{Algo:SMP} converges strongly to a solution $u\in\Omega$ of the problem SCMP \eqref{Split:MIN}, which uniquely solves the problem
\begin{align*}
\text{Find } u\in\Omega \text{ such that } \left\langle \mathscr{F}(u), x^*-u\right\rangle\geq 0 \; \quad\forall x^*\in\Omega.
\end{align*}
\end{theorem}

Similarly, using Remark \ref{remark:Weakalgo}, we deduce the following weak convergence theorem for the SCMP \eqref{Split:MIN} from Corollary \ref{Thm:Weakalgo}.

\begin{theorem}
Let $f:\mathscr{H}_1\rightarrow\mathbb{R}$ and $g:\mathscr{H}_2\rightarrow\mathbb{R}$ be two convex Fréchet differentiable functions. Let $A:\mathscr{H}_1\rightarrow \mathscr{H}_2$ be a bounded linear operator. Assume that the gradients $\nabla f$ and $\nabla g$ are Lipschitz continuous with constants $L_1$ and $L_2$, respectively. 
For a given $z_1\in \mathscr{H}_1$, compute
\begin{equation}\label{Algo:WSMP}
z_{n+1}=P_C(z_n-\lambda_n \nabla f(z_n)-\lambda_n A^*(I-P_Q(I-\lambda_n\nabla g))A(z_n)),
\end{equation}
where the parameters $\lambda_n$ satisfy the following condition:
\begin{itemize}
\item[(D1)] $0<c<\lambda_n<\frac{1}{\rho}min \Big\{\tilde{\tau},\frac{1}{\|A\|^2}\Big\}$ for each $n\in\mathbb{N}$, where $\tilde{\tau}=\min\{\frac{1}{L_1},\frac{1}{L_2}\}$ and $\rho>2$.
\end{itemize}
Then the sequence $\{z_n\}$ generated by the iterative algorithm \eqref{Algo:WSMP} converges weakly to a solution $u\in\Omega$ of the problem SCMP \eqref{Split:MIN}.
\end{theorem}

\subsection{Split variational inequality problem}

Let $C$ and $Q$ be two nonempty, closed and convex subsets of real Hilbert spaces $\mathscr{H}_1$ and $\mathscr{H}_2$, respectively.  Let $f_1:\mathscr{H}_1\rightarrow\mathscr{H}_1$ and  $f_2:\mathscr{H}_2\rightarrow\mathscr{H}_2$ be two single-valued mappings. Suppose $A:\mathscr{H}_1\rightarrow\mathscr{H}_2$ is a bounded linear operator. Then the {\em split variational inequality problem} (SVIP) \cite{YCEN22012} is defined as follows: find $x^*\in C$ such that
\begin{equation}\label{Split:VIP}
\begin{cases}
\left\langle f_1(x^*), x-x^*\right\rangle\geq0 \quad\forall x\in C\text{ and }\\
y^*=Ax^*\in Q \text{ solves }\left\langle f_2(y^*), y-y^*\right\rangle\geq 0 \quad\forall y\in Q.
\end{cases}
\end{equation}

If we take $B_1=\partial i_C$ and $B_2=\partial i_Q$, then the SVI (\ref{SVI}) reduces to the SVIP \eqref{Split:VIP}.
%That is, the problem \eqref{SVI} reduces to the following problem: Find $x^*\in C$ such that
%\begin{equation*}
%\begin{cases}
%\left\langle f_1(x^*), x-x^*\right\rangle\geq0, \quad\forall x\in C\\
%\text{and }y^*=Ax^*\in Q \text{ solves }\left\langle f_2(y^*), y-y^*\right\rangle\geq 0, \quad\forall y\in Q.
%\end{cases}
%\end{equation*}

Therefore, we can deduce the following strong convergence theorem for the SCMP \eqref{Split:MIN} from Theorem \eqref{Algo1} under some additional assumptions on the parameters and mappings involved.

We assume that the solution set $\Omega$ of the SVIP \eqref{Split:VIP} is nonempty.

\begin{theorem}
Let $f_1:\mathscr{H}_1\rightarrow\mathscr{H}_1$ and $f_2:\mathscr{H}_2\rightarrow\mathscr{H}_2$ be two single-valued mappings and $A:\mathscr{H}_1\rightarrow \mathscr{H}_2$ be a bounded linear operator. Suppose $(A2)$ and $(A3)$ hold. Let $z_1\in \mathscr{H}_1$ be arbitrary. Compute
\begin{align}\label{Algo:SVIP}
z_{n+1}=P_C(z_n-\lambda_n f_1(z_n)-\lambda_n A^*(I-P_Q(I-\lambda_n f_2))A(z_n)-\lambda_n\alpha_n\mathscr{F}(z_n)),
\end{align}
where the parameters satisfy conditions $(C1)$, $(C2)$, and $(B3)$.

Then the sequence $\{z_n\}$ generated by the iterative algorithm \eqref{Algo:SVIP} converges strongly to a solution $u\in\Omega$ of the SVIP \eqref{Split:VIP}, which uniquely solves the problem
\begin{align*}
\text{Find } u\in\Omega \text{ such that } \left\langle \mathscr{F}(u), x^*-u\right\rangle\geq 0 \; \quad\forall x^*\in\Omega.
\end{align*}
\end{theorem}

Similarly, using Remark \ref{remark:Weakalgo} we deduce the following weak convergence theorem for the SVIP \eqref{Split:VIP} from Corollary \ref{Thm:Weakalgo}.

\begin{theorem}
Let $f_1:\mathscr{H}_1\rightarrow\mathscr{H}_1$ and $f_2:\mathscr{H}_2\rightarrow\mathscr{H}_2$ be two single-valued mappings and let $A:\mathscr{H}_1\rightarrow \mathscr{H}_2$ be a bounded linear operator. Suppose $(A2)$ and $(A3)$ hold. Let $z_1\in \mathscr{H}_1$ be arbitrary. Compute
\begin{align}\label{Algo:wsvip}
z_{n+1}=P_C(z_n-\lambda_n f_1(z_n)-\lambda_n A^*(I-P_Q(I-\lambda_n f_2))A(z_n)),
\end{align}
where the parameters $\lambda_n $ satisfy $(B3)$.

Then the sequence $\{z_n\}$ generated by the iterative algorithm \eqref{Algo:wsvip} converges strongly to a solution $u\in\Omega$ of the SVIP \eqref{Split:VIP}.
\end{theorem}

%==============================================================================================================
\section{Numerical experiments}
\label{Sec:6}

\noindent
We give two examples to illustrate the performance of our algorithm \eqref{Algo1} and also compare it with \eqref{Algo2} and \eqref{algo:moudafi}. Our computations are performed using Matlab R2022a installed on a Desktop computer with Windows 8 and Intel(R) Core(TM) i5-3470 CPU @3.20GHZ and 8GB RAM.
\begin{example}
We consider $\mathscr{H}_1=\mathscr{H}_2=(\ell_2(\mathbb{R}), \|\cdot\|)$, where $\ell_2(\mathbb{R}) := \{x = (x_1, x_2, \dots, x_i, \dots), x_i\in \mathbb{R}: \sum\limits_{i=1}^{\infty}\vert x_i\vert^2<\infty\}$ with the associated norm $\|x\| = (\sum\limits_{i=1}^{\infty}\vert x_i\vert^2)^{\frac{1}{2}}$, $\forall x \in \ell_2(\mathbb{R})$. For each $x\in \ell_2(\mathbb{R})$, we define $\mathscr{B}_1x = 3x$, $\mathscr{B}_2x = 7x$, $f_1x=2x$, $f_2x = (x_1, \frac{x_2}{2}, \frac{x_3}{3}, \cdots)$, $\mathscr{F}x = 4x$, and $Ax = (x_1, x_1, \frac{x_2}{2},\frac{x_3}{3}, \cdots)$. We further choose $\alpha_n= \frac{3}{\sqrt{n}+3}$, $\lambda_n = \frac{n}{7n+3}$ for algorithms \eqref{Algo1} and \eqref{Algo2}, and $\lambda=0.1$ for \eqref{algo:moudafi}. For this experiment, we consider the following choices of the initial value $z_1$:\\
	\noindent {\it Case Ia:} $z_1 = (16, 4, 1, \frac{1}{4},\cdots)$;\\
	\noindent {\it Case Ib:}  $z_1 = (9, 3, 1, \frac{1}{3}, \cdots)$;\\
	\noindent {\it Case Ic:}  $z_1=(100, -10, 1,-0.1,\cdots);$\\
	\noindent {\it Case Id:} $z_1 = (-20, 4, -\frac{4}{5},\frac{4}{25},\cdots)$.
	
	\noindent
	The stopping criterion for the computations is $\text{Tol}_n \leq 10^{-6}$, where $\text{Tol}_n = \|z_n - J_{\lambda_n}^{\mathscr{B}_1}(z_n - f_1(z_n))\| + \|Az_n -J_{\lambda_n}^{\mathscr{B}_2}(Az_n - f_2(Az_n))\|$. Note that $\text{Tol}_n = 0$ means that $z_n$ solves the SVI (\ref{SVI}) . Figure \ref{fig1} and Table \ref{tab1} give the numerical results we obtained for the four different choices of the initial values.
	\begin{figure}	
		\begin{center}
			\includegraphics[height = 4cm]{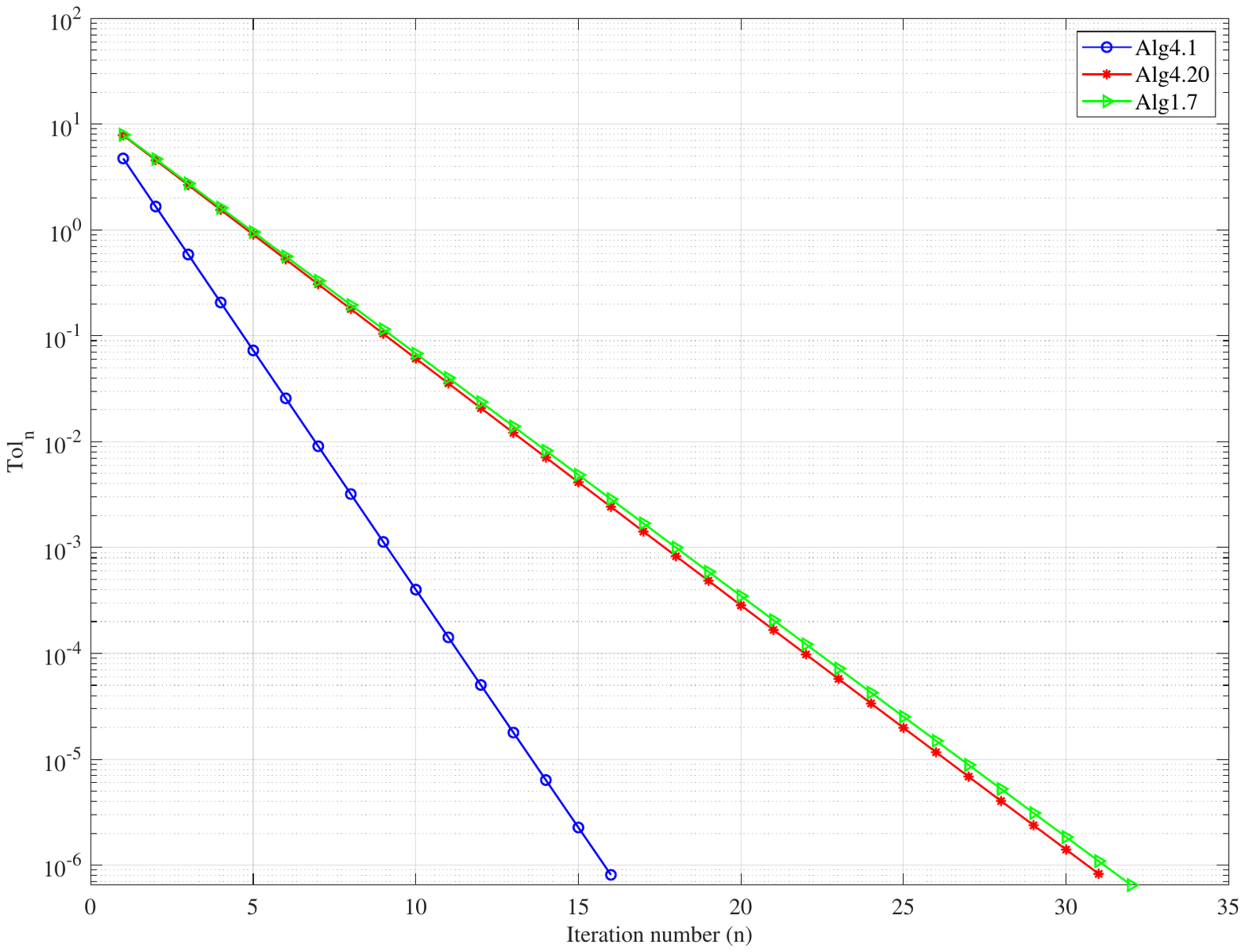}	
			\includegraphics[height = 4cm]{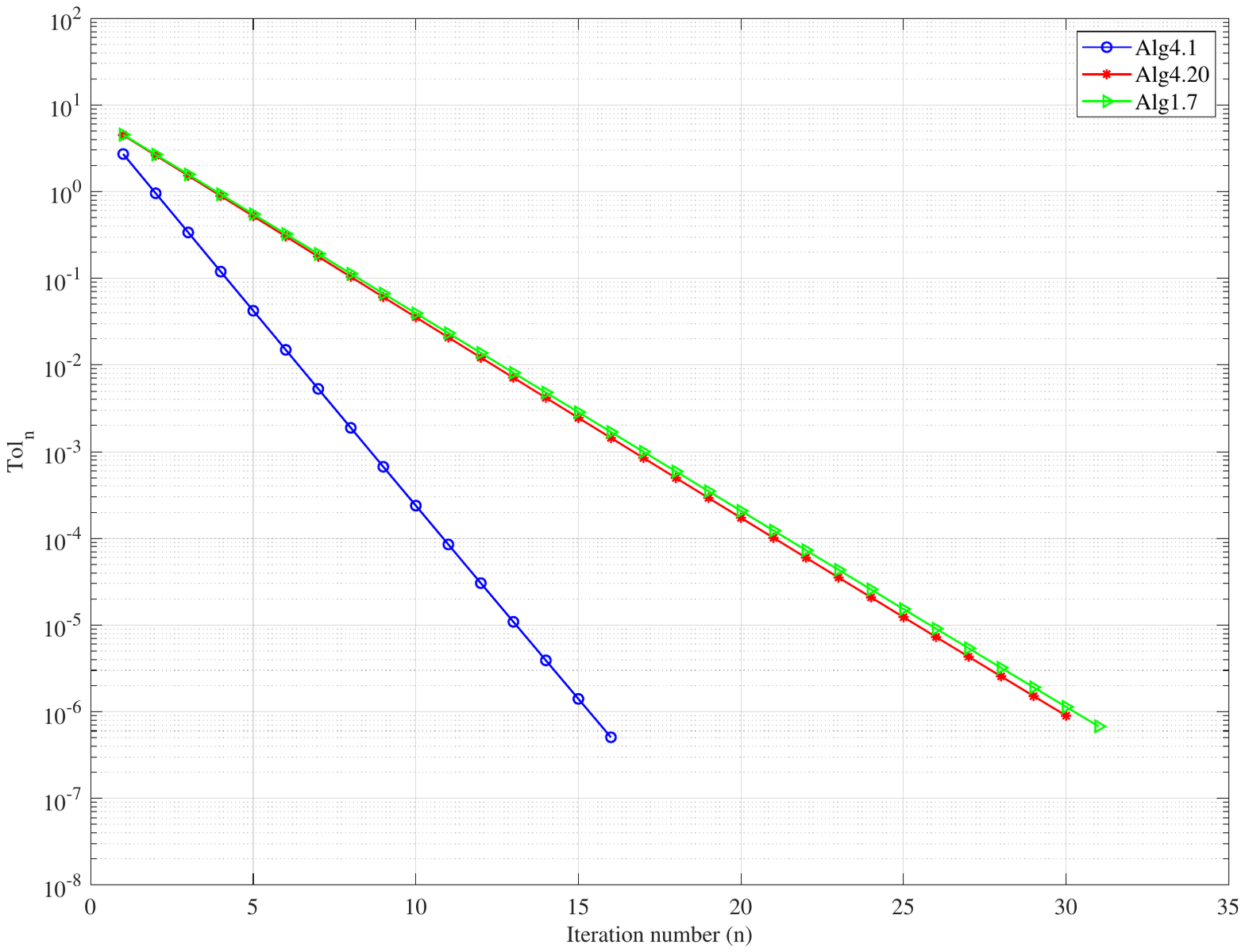}\\
			\includegraphics[height = 4cm]{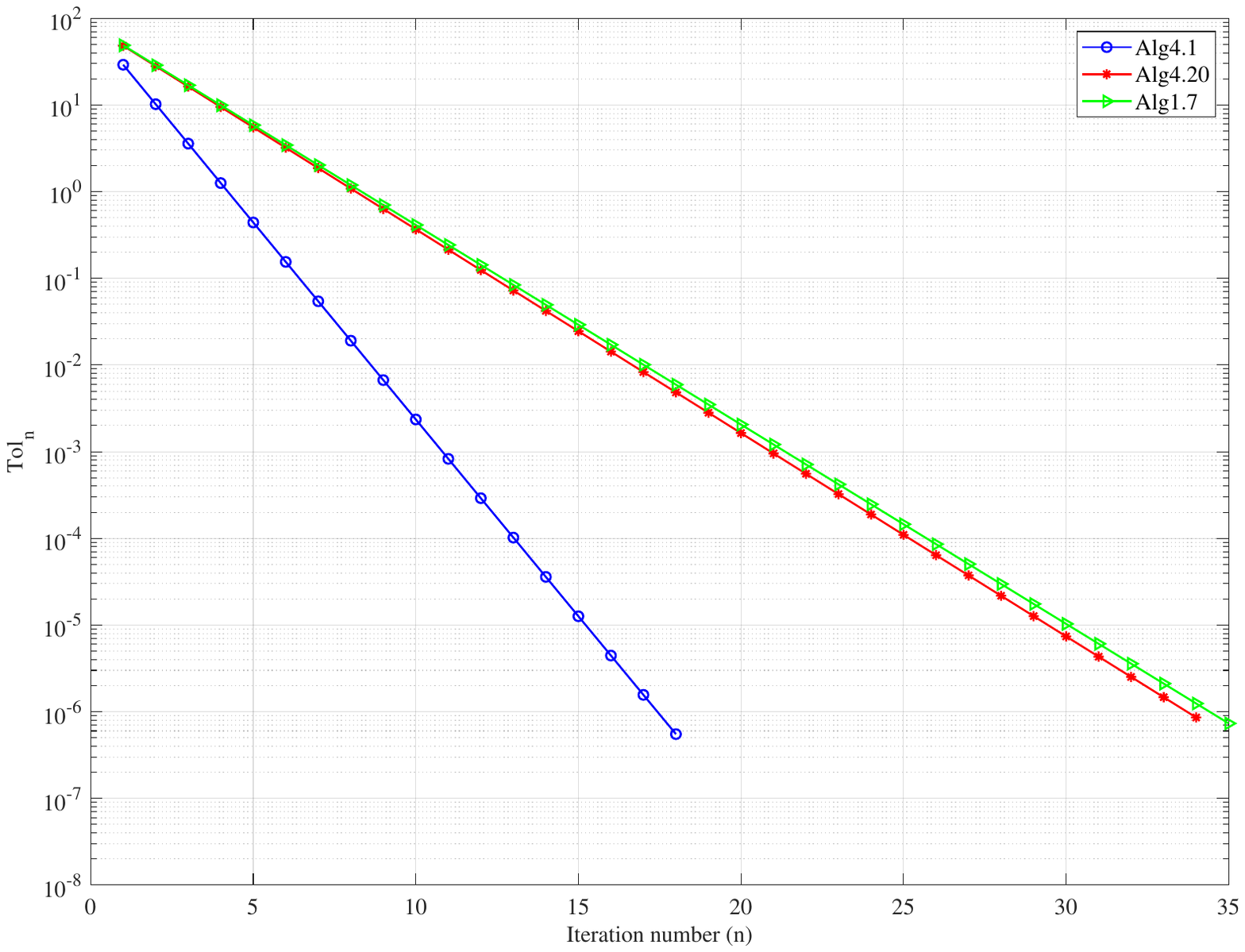}	
			\includegraphics[height = 4cm]{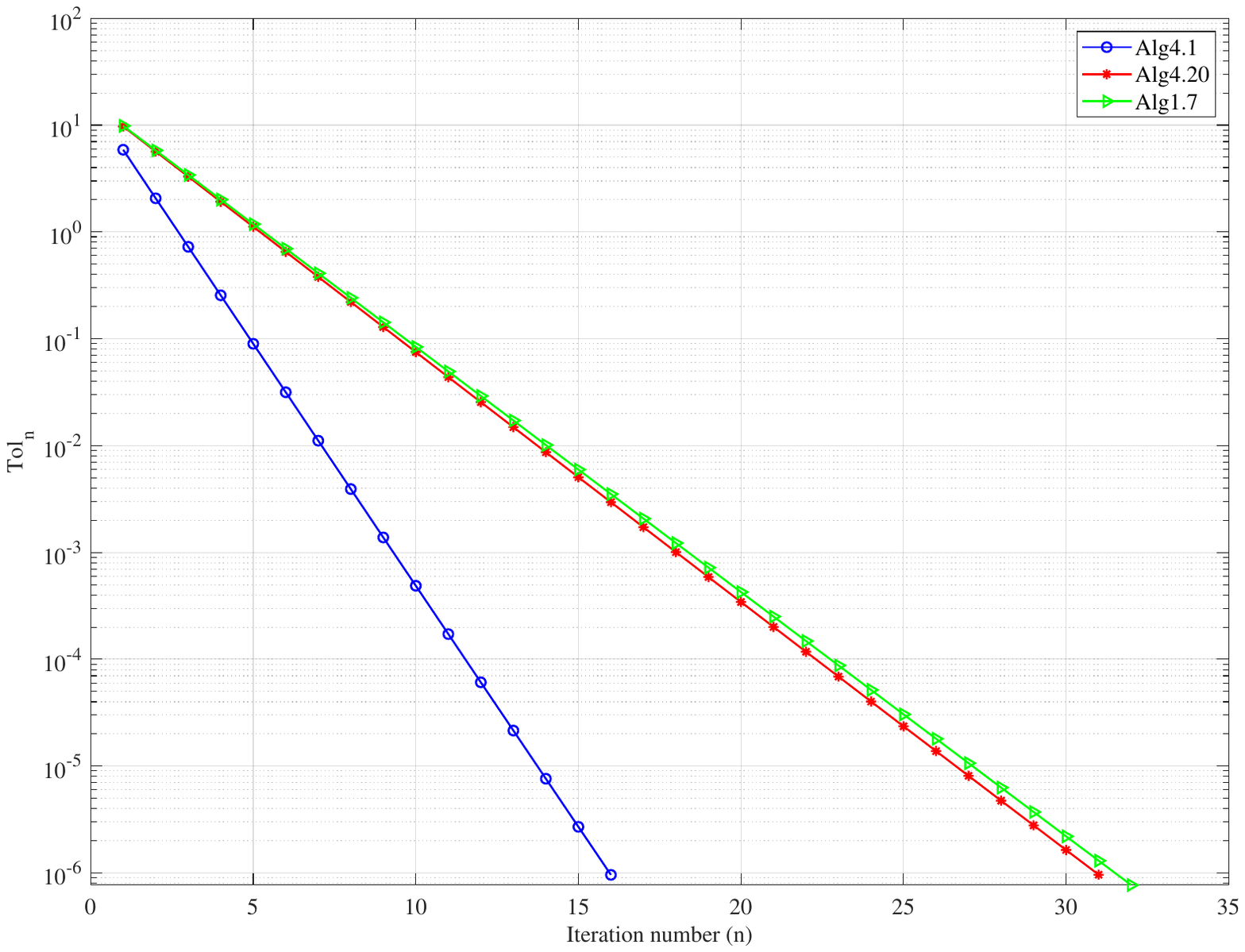}
		\end{center}
		\caption{ Top left: Case Ia; Top right: Case Ib; Bottom left: Case Ic; Bottom right: Case Id.}\label{fig1}
	\end{figure}	
	\begin{table}[h!]
		\caption{Numerical results.}
		\label{tab1}
		\begin{tabular}{ |p{1.4cm} |p{2.6cm}| p{1.8cm} | p{1.8cm}| p{1.8cm}| p{1.4cm}|p{1.4cm}|}
			\hline
			\noindent & \noindent & \eqref{Algo1} & \eqref{Algo2} &  \eqref{algo:moudafi} \\
			\hline
			Case Ia  &   CPU time (sec)  & 0.0036 & 0.0075 & 0.0013 \\
			& No of Iter.  & 16 & 31 & 32 \\
			\hline
			Case Ib &   CPU time (sec)  & 0.0035 & 0.0093 & 0.0016\\
			& No. of Iter. & 16 & 30 & 31\\
			\hline
			Case Ic & CPU time (sec)& 0.0040 &  0.0102 & 0.0013 \\
			& No of Iter. & 18 & 34 & 35 \\
			\hline
			Case Id & CPU time (sec)  & 0.0034 & 0.0106 & 0.0163 \\
			& No of Iter. & 16 & 31 & 32 \\
			\hline
		\end{tabular}
	\end{table}
\end{example}
%\begin{example}Let $\mathscr{H}_1 = \mathscr{H}_2 = L^2([0,1])$ endowed with the inner product $$\langle x, y \rangle:= \int_{0}^{1}x(t)y(t)dt \; \forall x, y \in L^2([0,1])$$ and the induced norm $$\|x\|:=\bigg(\int_{0}^{1}|x(t)|^2dt\bigg)^{\frac{1}{2}} \; \forall x\in L^2([0,1]), t\in [0,1].$$ Let $C\subset \mathscr{H}_1$ and $Q\subset \mathscr{H}_2$ be defined by $C:=\{x\in L^2([0,1]: \langle t^2, x\rangle = \frac{1}{2})\}$ and $Q:=\{x\in L^2([0,1]):\langle \frac{t}{2}, x\rangle = \frac{1}{6}\}$, respectively. Furthermore, we define $A:\mathscr{H}_1\to \mathscr{H}_2$ by $A(x)(t) = \frac{x(t)}{2}$, $f_1(x)(t) = 3x(t)$ and $f_2(x)(t) = 7x(t) - 2$. We intend to use these data to solve the SVIP \eqref{Split:VIP}. Note that in this case $x(t) = 2t\in \Omega$. We define $\mathscr{F}(x)(t):\mathscr{H}_1\to \mathscr{H}_1$ by $\mathscr{F}(x)(t) = 2x(t)$.  We further choose $\alpha_n= \frac{1}{\sqrt{n}+3}$, $\lambda_n = \frac{1}{7n+3}$ for algorithms \eqref{Algo1} and \eqref{Algo2}, and $\lambda=0.1$ for \eqref{algo:moudafi}. For this experiment, we consider the following choices of the initial value $x_1$:\\
%	\noindent {\it Case IIa:} $x_1 = t^2+1$;\\
%	\noindent {\it Case IIb:}  $x_1 = \sin 2t$;\\
%	\noindent {\it Case IIc:}  $x_1=t^2-2t+3;$\\
%	\noindent {\it Case IId:} $x_1 = 2\cos 2t$.
	
%\end{example}

\begin{example}In this experiment, we intend to illustrate the application in Section \ref{Sec:5.1} numerically. Let $E:\mathscr{H}_1\to \mathbb{R}$ and $G:\mathscr{H}_2\to \mathbb{R}$ be defined by $E(x) = \|x\|^2 + (1,1,-3)\cdot x+ 2+\|x\|_1$ and $G(x) = \|x\|^2 + (1,1,-5)\cdot x-3 + \|x\|_1$, respectively. Furthermore, we define $A:\mathscr{H}_1\to \mathscr{H}_2$ by $Ax = 2x$. Note that in this case $x^* = (0,0,1)\in \Omega\neq \emptyset$. We define $\mathscr{F}:\mathscr{H}_1\to \mathscr{H}_1$ by $\mathscr{F}x = 2x$.  We further choose $\alpha_n= \frac{0.01}{\sqrt{500n+2}+2}$, $\lambda_n = \frac{n}{14n+1}$ for algorithms \eqref{Algo1} and \eqref{Algo2}, and $\lambda=1/15$ for \eqref{algo:moudafi}. For this experiment, we consider the following choices of the initial value $z_1$:\\
	\noindent {\it Case IIa:} $z_1 = (1, -2, 16)$;\\
	\noindent {\it Case IIb:}  $z_1 = (15, 9, 0)$;\\
	\noindent {\it Case IIc:}  $z_1=(1, 0, 6);$\\
	\noindent {\it Case IId:} $z_1 = (11, 1, -3)$.
	
	\noindent
	The stopping criterion for the computations is $\text{Tol}_n \leq 10^{-4}$, where $\text{Tol}_n = \|z_n - (0,0,1)\| + \|Az_n - (0,0,2)\|$. Note that $\text{Tol}_n = 0$ means that $z_n$ solves the SCMP (\ref{Split:MIN}). Figure \ref{fig2} and Table \ref{tab2} give the numerical results we obtained for the four different choices of the initial values.
\begin{figure}	
	\begin{center}
		\includegraphics[height = 4cm]{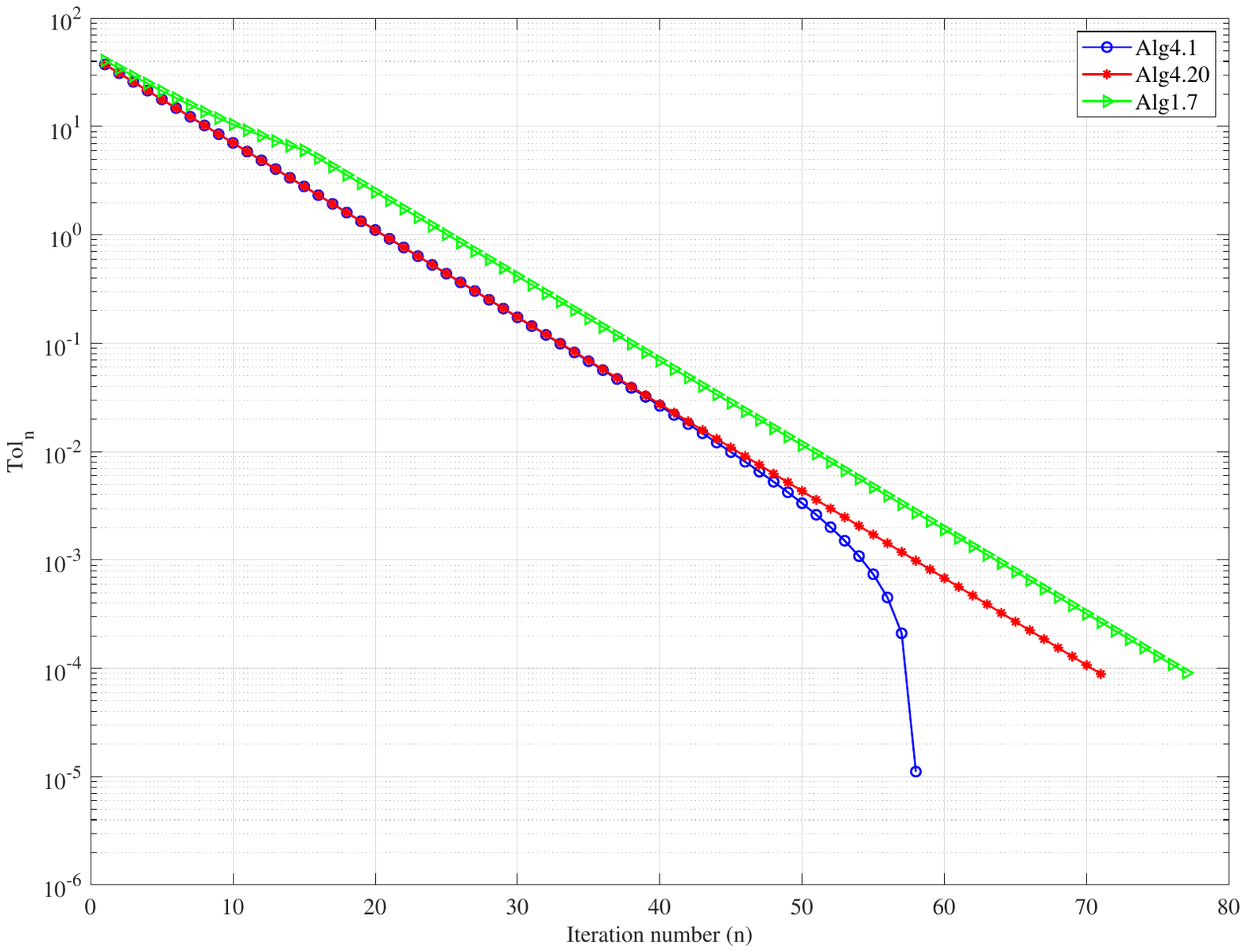}	
		\includegraphics[height = 4cm]{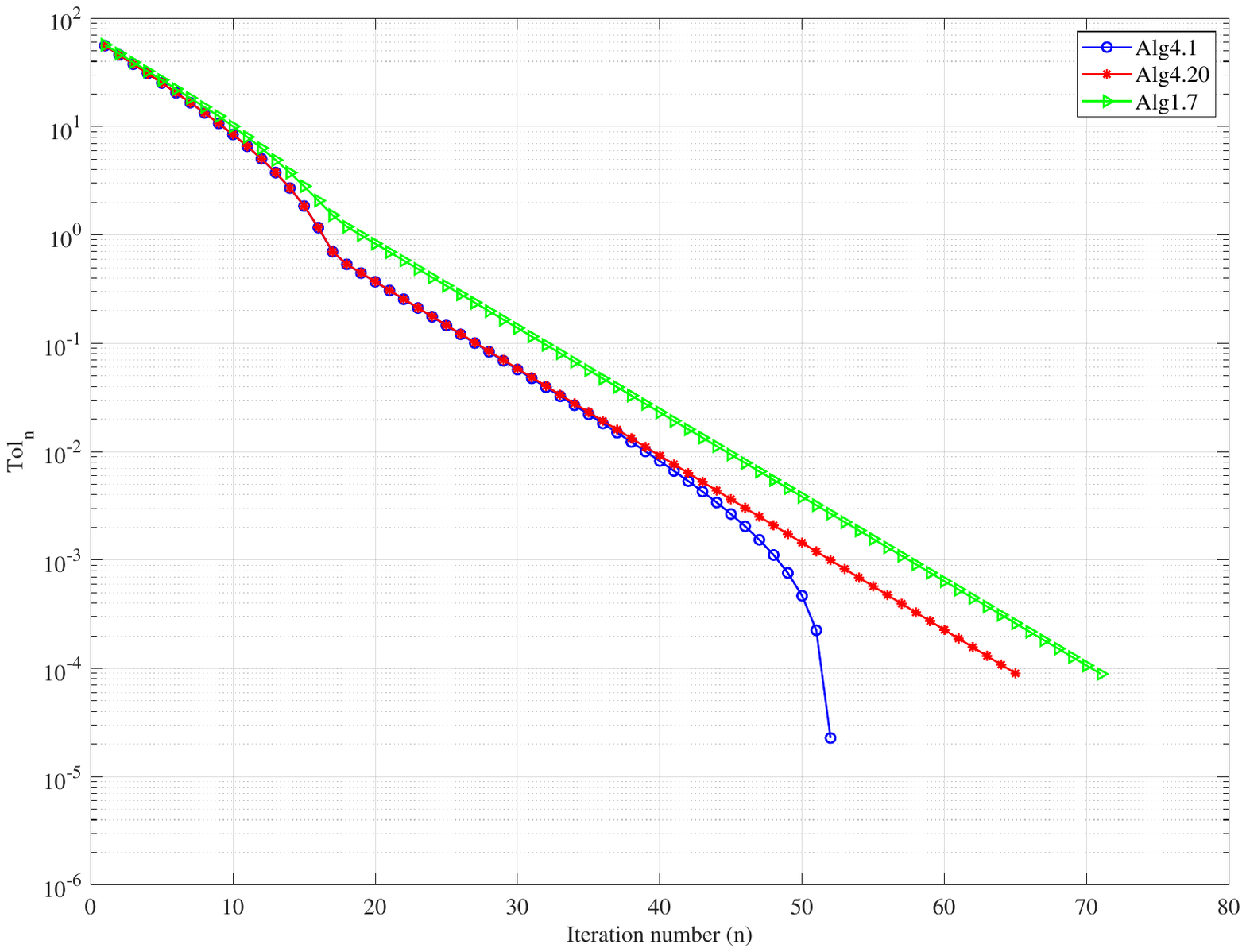}\\
		\includegraphics[height = 4cm]{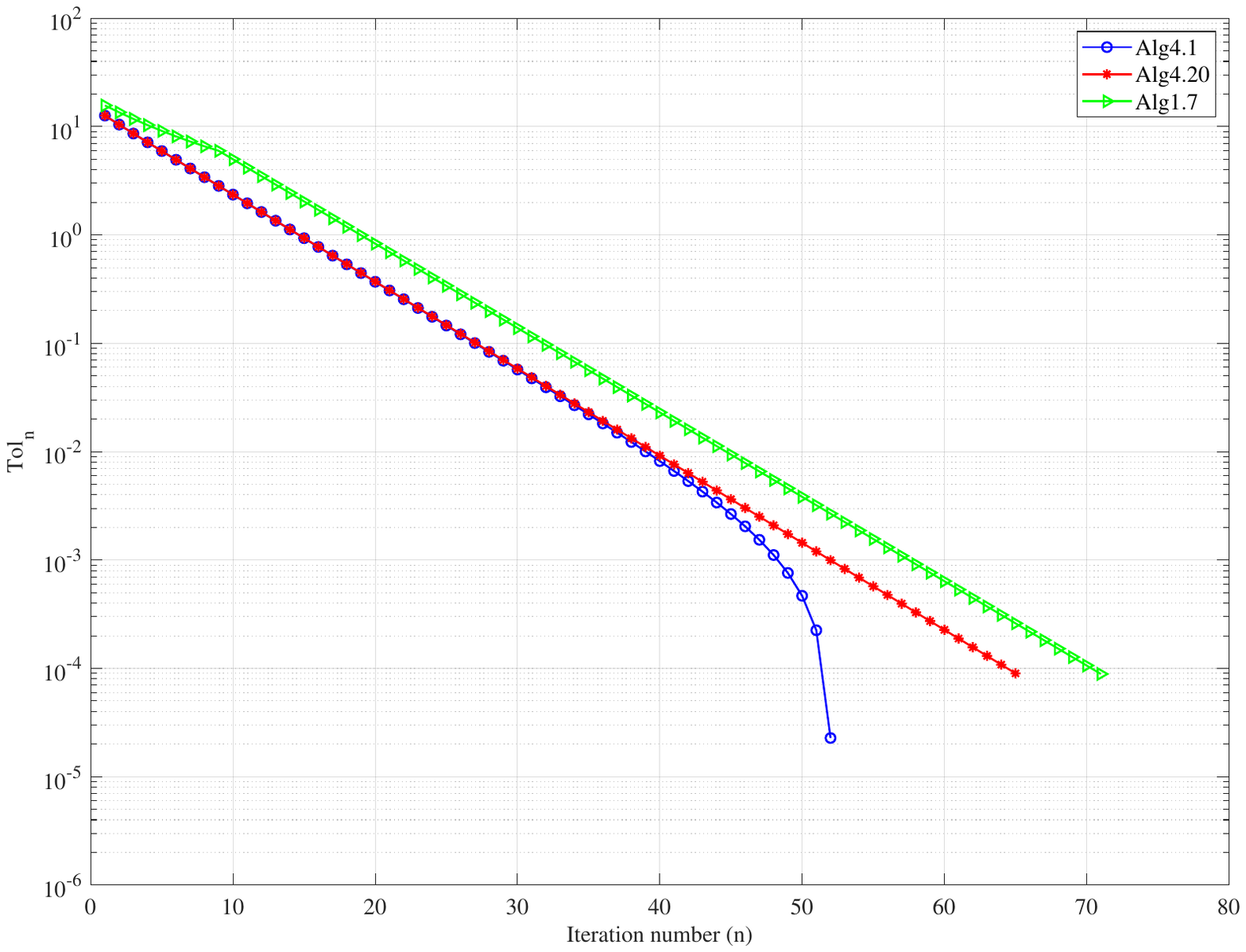}	
		\includegraphics[height = 4cm]{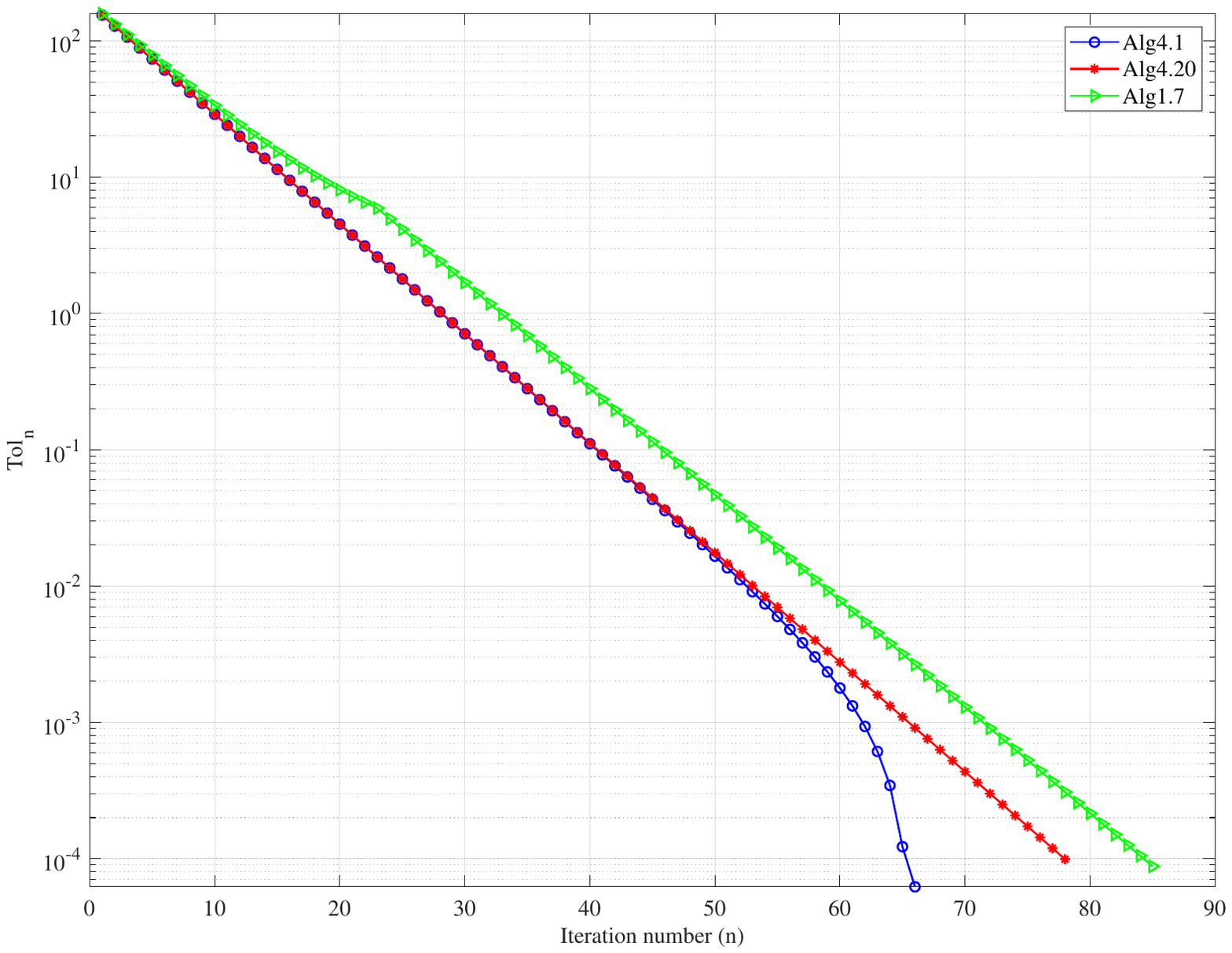}
	\end{center}
	\caption{Top left: Case IIa; Top right: Case IIb; Bottom left: Case IIc; Bottom right: Case IId.}\label{fig2}
\end{figure}	
\begin{table}[h!]
\caption{Numerical results.}
\label{tab2}
\begin{tabular}{ |p{1.4cm} |p{2.6cm}| p{1.8cm} | p{1.8cm}| p{1.8cm}| p{1.4cm}|p{1.4cm}|}
	\hline
	\noindent & \noindent & \eqref{Algo1} & \eqref{Algo2} &  \eqref{algo:moudafi} \\
	\hline
	Case IIa  &   CPU time (sec)  & 8.7847e-04 & 8.8167e-04 & 9.6891e-04  \\
	& No of Iter.  & 58 & 71 & 77 \\
	\hline
	Case IIb &   CPU time (sec)  & 0.0011 & 0.0013 & 0.0017\\
	& No. of Iter. & 52 & 65 & 71\\
	\hline
	Case IIc & CPU time (sec)& 0.0012 &  0.0015 & 0.0075 \\
	& No of Iter. & 52 & 65 & 71 \\
	\hline
	Case IId & CPU time (sec)  & 9.8431e-4 & 9.5435e-04& 0.0010 \\
	& No of Iter. & 66 & 78 & 85 \\
	\hline
\end{tabular}
\end{table}
	
\end{example}

%
%%=======================================================================================================

%\section*{Acknowledgments}
%Soumitra Dey and Adeolu Taiwo gratefully acknowledges the financial support of the Post-Doctoral Program at the Technion - Israel Institute of Technology. Simeon Reich was partially supported by the Israel Science Foundation (Grant 820/17), by the Fund for the Promotion of Research at the Technion (Grant 2001893) and by the Technion General Research Fund (Grant 2016723).

\section*{Disclosure statement}
No potential conflict of interest was reported by the authors.

\section*{Data availability statement}
The authors acknowledge that the data presented in this study must be deposited and made
publicly available in an acceptable repository, prior to publication.

%===========================================

\end{document}